\documentclass[11pt]{amsart}

\usepackage{hyperref}
\usepackage{mathrsfs,dsfont,bbm}
\usepackage{amsmath,latexsym,amssymb,amsthm,array,amsfonts}
\usepackage{color}

\numberwithin{equation}{section}

\newtheorem{theorem}{Theorem}[section]
\newtheorem{lemma}{Lemma}[section]
\newtheorem{corollary}{Corollary}[section]
\newtheorem{proposition}{Proposition}[section]

\newtheorem{remark}{Remark}[section]

\begin{document}

\title[An integral functional driven by fractional Brownian motion]{An integral functional driven by fractional Brownian motion$^{*}$}

\footnote[0]{$^{*}$The Project-sponsored by NSFC (11571071, 11426036) and Innovation Program of Shanghai Municipal Education Commission(12ZZ063).

${}^\S$litan-yan@hotmail.com (Corresponding author)}

\author[X. Sun, L. Yan and X. Yu]{Xichao Sun${}^{1}$, Litan Yan${}^{2,\S}$ and Xianye Yu${}^{2}$}



\maketitle

\date{}

\begin{center}
{\footnotesize {\it  ${}^1$Department of Mathematics and Physics, Bengbu University\\
1866 Caoshan Rd., Bengbu 233030, P.R. China\\
${}^2$Department of Mathematics, College of Science, Donghua University\\
2999 North Renmin Rd., Songjiang, Shanghai 201620, P.R. China}}
\end{center}


\maketitle

\begin{abstract}
Let $B^H$ be a fractional Brownian motion with Hurst index
$0<H<1$ and the weighted local time ${\mathscr L}^H(\cdot,t)$. In this paper, we consider the integral functional
$$
{\mathcal C}^H_t(a):=\lim_{\varepsilon\downarrow 0}\int_0^t1_{\{|B^H_s-a|>\varepsilon\}}\frac1{B^H_s-a}ds^{2H}\equiv \frac1{\pi}{\mathscr H}{\mathscr L}^H(\cdot,t)(a)
$$
in $L^2(\Omega)$ with $ a\in {\mathbb R}, t\geq 0$ and ${\mathscr H}$ denoting the Hilbert transform. We show that
$$
{\mathcal C}^H_t(a)=2\left((B^H_t-a)\log|B^H_t-a|-B^H_t+a\log|a| -\int_0^t\log|B^H_s-a|\delta B^H_s\right)
$$
for all $a\in {\mathbb R}, t\geq 0$ which is the fractional version of Yamada's formula, where the integral is the Skorohod integral. Moreover, we introduce the following {\em occupation type formula}:
$$
\int_{\mathbb R}{\mathcal C}^H_t(a)g(a)da=2H\pi\int_0^t({\mathscr
H}g)(B^H_s)s^{2H-1}ds
$$
for all continuous functions $g$ with compact support.
\end{abstract}


\section{Introduction}
Given $H\in (0,1)$, a fractional Brownian motion (fBm) $B^H=\{B_t^H, 0\leq t\leq
T\}$ with Hurst index $H$ is a mean zero Gaussian process such that
$$
E\left[B_t^HB_s^H\right]=\frac{1}{2}\left[t^{2H}+s^{2H}-|t-s|^{2H}\right]
$$
for all $t,s\geq 0.$ For $H=1/2$, $B^H$ coincides with the
standard Brownian motion $B$. $B^H$ is neither a semimartingale nor
a Markov process unless $H=1/2$, so many of the powerful techniques
from stochastic analysis are not available when dealing with $B^H$.
As a Gaussian process, one can construct the stochastic calculus of
variations with respect to $B^{H}$. Some surveys and complete
literatures for fBm could be found in Biagini {\it et
al}~\cite{BHOZ}, Decreusefond-\"Ust\"unel~\cite{Dec}, Hu~\cite{Hu2},
Mishura~\cite{Mishura2}, Nourdin~\cite{Nourdin}, Nualart~\cite{Nua4} and the references therein.

Let now $F$ be an absolutely continuous function such that the Skorohod integral
$$
\int_0^tF'(B^H_s-a)\delta B^H_s
$$
is well-defined and the second derivative $F''=f$ exists in the sense of Schwartz's distribution. Then the process
\begin{equation}\label{sec1-eq1.1}
{\mathcal K}^H_t(a):=2\left(F(B^H_t-a)-F(-a)-\int_0^tF'(B^H_s-a)\delta B^H_s\right),
\end{equation}
exists for all $a\in {\mathbb R}$. Denote
\begin{equation}\label{sec1-eq1.2}
{\mathcal X}^H_t(a):=\int_0^tf(B^H_s-a)ds^{2H}
\end{equation}
for all $t\geq 0,a\in {\mathbb R}$. By It\^o's formula one can find the following questions:
\begin{itemize}
\item if the Lebesgue integral~\eqref{sec1-eq1.2} converges,
$$
{\mathcal K}^H_t(a)={\mathcal X}^H_t(a)\;?
$$
\item how to characterize the process ${\mathcal K}^H_t(a)$ if the Lebesgue integral~\eqref{sec1-eq1.2} diverges?
\end{itemize}

Clearly, the first question is positive by approximating. However, the second question is not obvious even if $H=\frac12$ and $f$ is a special function. Thus, the question arises again:
\begin{itemize}
\item for which functions does the Lebesgue integral~\eqref{sec1-eq1.2} diverge?
\end{itemize}

When $H=\frac12$, $B^H$ coincides with the standard Brownian motion $B$ and by the Engelbert-Schmidt zero-one law, the Lebesgue integral~\eqref{sec1-eq1.2} diverges if $F''=f$ is not locally integrable, i.e.
$$
\int_{-M}^M|f(x-a)|dx=\infty
$$
for some $M>0$. Thus, when
$$
|f(x)|\geq C|x|^{-\alpha}
$$
for $x\in {\mathbb R}$, the Lebesgue integral~\eqref{sec1-eq1.2} diverges, where $C>0,\alpha\geq 1$. For some special functions $F$, for examples,
\begin{equation}\label{sec1-eq1.3}
F''(x)=|x|^{-\gamma}{\rm sign}(x)
\end{equation}
with $1\leq \gamma<\frac32$, one studied the characterization and properties of the process ${\mathcal K}^{\frac12}_t(a)$. It\^o--McKean~\cite{Ito-McKean} first considered the process ${\mathcal K}^{\frac12}_t(a)$ and the Lebesgue integral~\eqref{sec1-eq1.2} for $F$ satisfying~\eqref{sec1-eq1.3} with $\gamma=1$.
For the process ${\mathcal K}^{\frac12}_t(a)$ and the Lebesgue integral~\eqref{sec1-eq1.2} driven by the function $F$ satisfying~\eqref{sec1-eq1.3}, some systematic studies are due to Biane-Yor~\cite{Biane-Yor}, Yamada~\cite{Yamada1,Yamada2,Yamada4}
and Yor~\cite{Yor1}, and some extensions and limit theorems are established by Bertoin~\cite{Bertoin1,Bertoin2},
Cherny~\cite{Cherny}, Csaki {\em et al }~\cite{Csaki1,Csaki2},
Csaki-Hu~\cite{Hu-Y}, Hu~\cite{Hu-Y}, Fitzsimmons-Getoor~\cite{Fitzsimmons-Getoor1,Fitzsimmons-Getoor2}, Mansuy-Yor~\cite{Mansuy-Yor}, Yor~\cite{Yor2} and the references therein. However, those researches apply only to Markov process, and for non-Markov processes there has only been little investigation on the integral functional. See Eddahbi-Vives~\cite{Eddahbi-Vives}, Gradinaru {\em et al.}~\cite{Grad2}, Yan~\cite{Yan8} and Yan-Zhang~\cite{Yan10}.

When $H\neq \frac12$ the second and third questions above are not trivial. The main difficulty consists in the fact that the stochastic integral
\begin{equation}\label{sec1-eq1.4}
\int_0^tF'(B^H_s-a)\delta B^H_s
\end{equation}
is a Skorohod integral with respect to the fBm and the integrand is not smooth. Therefore, its control is not obvious and one needs sharp estimates. The $L_2$ norm of this stochastic Skorohod integral involves the Malliavin derivatives of the integrand and tedious estimation on the joint density of the fBm. Moreover, for a nonsmooth function $f$ it is not easy to give an exact calculus of the moment of order $2$ for the Skorohod integral~\eqref{sec1-eq1.4} even if the simple functions $F'(x)=\log|x|$ and $F'(x)=|x|^{-\alpha}{\rm sign}(x)$ with
$\alpha>0$. But, when $H=\frac12$, the integral~\eqref{sec1-eq1.4} is It\^o's integral and its existence is obvious. On the other hand, it is unclear whether the Engelbert-Schmidt zero-one law actually holds for fBm $B^H$. Thus, it seems interesting to study the process ${\mathcal K}^H_t(a)$ and the Skorohod integral~\eqref{sec1-eq1.4} with the singular integrand for $H\neq \frac12$. In this paper, as a start reviewing the object and continued to Yan~\cite{Yan8}, we consider the integrals
\begin{equation}\label{sec1-eq1.5}
\int_0^t\frac{ds^{2H}}{B^H_s-a},\qquad a\in {\mathbb R}
\end{equation}
and the processes
\begin{equation}\label{sec1-eq1.6}
{\mathcal C}^H_t(a):=2\left(
F(B^H_t-a)-F(-a)-\int_0^tF'(B^H_s-a)\delta B^H_s\right),\qquad a\in {\mathbb R}
\end{equation}
with $t\geq 0$, where the integral in~\eqref{sec1-eq1.6} is the Skorohod integral and $F(x)=x\log{|x|}-x$. In the present paper we will consider the functional and discuss some related questions. We will divide the discussion as two parts since the research method of the case $\frac12<H<1$ is essentially different with the case $0<H<\frac12$. In Section~\ref{sec7} we study the case $0<H<\frac12$ and the case $\frac12<H<1$ is considered in Section~\ref{sec3}, Section~\ref{sec4} and Section~\ref{sec5}.

This paper is organized as follows. In Section~\ref{sec2} we present some preliminaries for fBm. In Section~\ref{sec3}, we consider the existence of ${\mathcal C}^H(a)$ for $\frac12<H<1$. In fact, by smoothness approximating one can prove the existence of the Skorohod integral
$$
\int_0^t\log|B^H_s-a|\delta B^H_s,
$$
however, it is not easy to give the exact estimates of the moments. To give the existence and exact estimates of the moments, we define the function
\begin{align*}
\Psi_{s,r,a,b}(x,y):=\varphi_{s,r}&(x,y)-\varphi_{s,r}(x,b)
\theta(1+b-y)\\
&-\varphi_{s,r}(a,y)\theta(1+a-x)
+\varphi_{s,r}(a,b)\theta(1+a-x)\theta(1+b-y)
\end{align*}
with $x,y,a,b\in {\mathbb R},s,r>0$, where $\theta(x)=1_{\{x>0\}}$ and $\varphi_{s,r}(x,y)$ is the density function of $(B^H_s,B^H_r)$, and show that the identity
\begin{equation}\label{sec1-eq1.8}
\begin{split}
E\left[G'_1(B^H_s-a)G'_2(B^H_r-b)\right]&=\int_{{\mathbb
R}^2} G'_1(x-a)G'_2(y-b)\Psi_{s,r,a,b}(x,y)dxdy
\end{split}
\end{equation}
holds for all $G_1,G_2\in C^\infty({\mathbb R})$ with compact supports and $G_1(1)=G_2(1)=0$. By using~\eqref{sec1-eq1.8} we show that the integral $\int_0^tF'_{+}(B^H_s-a)\delta B^H_s$ exists and
\begin{align*}
E\left|\int_0^tF'_{+}(B^H_s-a)\delta B^H_s\right|^2&
=\int_0^t\int_0^tE\left[F'_{+}(B^H_s-a)F'_{+}(B^H_r-a) \right]
\phi(s,r)dsdr\\
&\hspace{-2.5cm}+\int_0^tds\int_0^sd\xi\int_0^tdr\int_0^rd\eta
\phi(s,\eta)\phi(r,\xi)\int_a^{\infty}\int_a^\infty
\frac{\Psi_{s,r,a,a}(x,y)}{(x-a)(y-a)}dxdy
\end{align*}
with $\phi(s,r)=H(2H-1)|s-r|^{2H-2}$, where the integral $\int_0^{\cdot}F'_{+}(B^H_s)\delta B^H_s$ is the Skorohod
integral and
$$
F_{+}(x)=
\begin{cases}
0,& {\text {if $x\leq 0$}},\\
x\log{x}-x,& {\text {if $x>0$}}.
\end{cases}
$$
In Section~\ref{sec4}, for $\frac12<H<1$ we show that the representation
\begin{equation}\label{sec1-eq1.9}
{\mathcal C}^{H}_t(a)=\lim_{\varepsilon\downarrow 0}\int_0^t1_{\{
|B^H_s-a|\geq \varepsilon \}}\frac{ 2Hs^{2H-1}}{B^H_s-a}ds,\quad a\in {\mathbb R},t\geq 0
\end{equation}
holds in $L^2(\Omega)$, which points out that $a\mapsto \frac1{\pi}\,{\mathcal C}^H_t(a)$ coincides with the Hilbert transform of the weighted local time
$$
a\mapsto {\mathscr L}^H(a,t)=2H\int_0^t\delta(B^H_s-a)s^{2H-1}ds
$$
and the fractional version of Yamada's formula
$$
(B^H_t-a)\log|B^H_t-a|-(B^H_t-a)=-a\log|a|+a+\int_0^t\log|B^H_s-a|
\delta B^H_s+\frac12{\mathcal C}^{H}_t(a)
$$
holds. In section~\ref{sec5} we introduce the so-called {\em occupation type formula}
\begin{equation}\label{sec1-eq1.10}
\int_{\mathbb R}{\mathcal C}^H_t(a)g(a)da=2H\pi\int_0^t({\mathscr
H}g)(B^H_s)s^{2H-1}ds
\end{equation}
for all continuous function $g$ with compact support and $\frac12<H<1$, where ${\mathscr H}$ denotes Hilbert transform. In Section~\ref{sec7} we study the case $0<H<\frac12$ by using the {\it generalized quadratic covariation} introduced in Yan {\em et al}~\cite{Yan7}.


\section{Preliminaries}\label{sec2}


\subsection{Cauchy principal value}

It is known that the Cauchy principal value, named after Augustin Louis Cauchy, is a method for assigning values to certain improper integrals which would otherwise be undefined. Depending on the type of singularity in the integrand $f$, the Cauchy principal value is defined as one of the following:
\begin{align}\label{sec2-eq2000}
\lim_{\varepsilon\downarrow 0}\left(\int_a^{c-\varepsilon}f(x)dx+\int_{c+\varepsilon}^bf(x)dx \right)=\lim_{\varepsilon\downarrow 0}\int_a^b1_{\{|c-x|\geq \varepsilon\}}f(x),
\end{align}
where $c\in (a,b)$ is a unique point such that
$$
\int_a^bf(x)dx=\infty.
$$
The limiting operation given in~\eqref{sec2-eq2000} is called the (Cauchy) principal value of the divergent integral $\int_a^bf(x)dx$ and the limiting process displayed in~\eqref{sec2-eq2000} is denoted as
$$
{\rm v.p.}\int f(x)dx.
$$
The notation {\rm v.p.} (valeur principale) is seen in European writings. We have, as an example,
$$
{\rm v.p.}\int_a^b\frac{dx}{c-x}=\log\frac{c-a}{b-c}
$$
for all $a<c<b$. Moreover, for a Borel function $\varphi\,:\,{\mathbb R}_{+}\to {\mathbb R}$ with
$$
\int_a^{a+1}\frac{|\varphi(x)-\varphi(a)|}{x-a}dx
+\int_{a+1}^\infty\frac{|\varphi(x)|}{x-a}dx<\infty,
$$
we can define the Cauchy's principal value
\begin{align*}
{\rm v.p.}\int_a^\infty\frac{\varphi(x)}{x-a}dx:&=\int_a^{a+1}
\frac{\varphi(x)-\varphi(a)}{x-a}dx+\int_{a+1}^\infty
\frac{\varphi(x)}{x-a}dx\\
&=\lim_{\varepsilon\downarrow
0}\left(\int_{a+\varepsilon}^\infty\frac{\varphi(x)}{x-a}dx
+\varphi(a)\log{\varepsilon}\right).
\end{align*}

Recall that the Hilbert transform ${\mathscr H}f$ of $f\in L^2({\mathbb R})$ is defined as follows
\begin{equation}\label{sec3-eq5.23}
{\mathscr H}f(a):=\frac1\pi\lim_{\varepsilon\downarrow
0}\int_{\mathbb R}1_{\{|x-a|\geq
\varepsilon\}}\frac{f(x)dx}{x-a}=\frac1\pi{\rm v.p.}\int_{\mathbb R}\frac{f(x)dx}{x-a}\equiv \frac1\pi\;{\rm v.p.}\frac1x\ast f(x),
\end{equation}
where $\ast$ denotes the convolution in the theory of distributions, which plays an important role in real and complex analysis. It is also important to note that ${\mathcal H}f$ belongs to $L^2$ and
$$
\int_{\mathbb R}({\mathscr H}f(x))^2dx=\int_{\mathbb R}f^2(x)dx
$$
holds, and moreover, if $f$ is a H\"older continuous function with compact support, then the limit in~\eqref{sec3-eq5.23} exists for every $x\in {\mathbb R}$. For more aspects on these material we refer to King~\cite{King}.

\subsection{Fractional Brownian motion}

In this subsection, we briefly recall some basic definitions and
results of fBm. For more aspects on these material we refer to Al\'os et al.~\cite{Nua1}, Biagini {\it et al}~\cite{BHOZ}, Decreusefond-\"Ust\"unel~\cite{Dec}, Hu~\cite{Hu2},
Mishura~\cite{Mishura2}, Nourdin~\cite{Nourdin}, Nualart~\cite{Nua4} and the references therein. Throughout this paper we assume that $0<H<1$ is arbitrary but fixed and we let $B^H=\{B_t^H, 0\leq t\leq T\}$ be a one-dimensional
fBm with Hurst index $H$ defined on $(\Omega, \mathcal{F}^H, P)$.

Let $\mathcal H$ be the completion of the linear space ${\mathcal E}$ generated by the indicator functions ${1}_{[0,t]},
t\in [0,T]$ with respect to the inner product
$$
\langle {1}_{[0,s]},{1}_{[0,t]}
\rangle_{\mathcal H}=\frac{1}{2}\left[t^{2H}+s^{2H}-|t-s|^{2H}
\right].
$$
The application $\varphi\in {\mathcal E}\to B^{H}(\varphi)$ is an
isometry from ${\mathcal E}$ to the Gaussian space generated by
$B^{H}$ and it can be extended to ${\mathcal H}$. When $\frac12<H<1$ the Hilbert space ${\mathcal H}$ can be written as
$$
{\mathcal H}=\left\{\varphi:[0,T]\to {\mathbb
R}\;\;|\;\;\|\varphi\|_{\mathcal H}<\infty\right\},
$$
where
$$
\|\varphi\|^2_{\mathcal
H}:=\int^T_0\int^T_0\varphi(s)\varphi(r) \phi(s,r)dsdr
$$
with $\phi(s,r)=H(2H-1)|s-r|^{2H-2}$. Notice that the elements of the
Hilbert space ${\mathcal H}$ may not be functions but distributions
of negative order (see, for instance, Pipiras-Taqqu~\cite{Pipiras}). Denote by $\mathcal S$ the set of smooth functionals of the form
\begin{equation}\label{sec2-eq2.0}
F=f(B^{H}(\varphi_1),B^{H}(\varphi_2),\ldots,B^{H}(\varphi_n)),
\end{equation}
where $f\in C^{\infty}_b({\mathbb R}^n)$ ($f$ and all its
derivatives are bounded) and $\varphi_i\in {\mathcal H}$. The {\it
derivative operator} $D^{H}$ (the Malliavin derivative) of a
functional $F$ of the form~\eqref{sec2-eq2.0} is defined as
$$
D^{H}F=\sum_{j=1}^n\frac{\partial f}{\partial
x_j}(B^{H}(\varphi_1),B^{H}(\varphi_2),
\ldots,B^{H}(\varphi_n))\varphi_j.
$$
The derivative operator $D^{H}$ is then a closable operator from
$L^2(\Omega)$ into $L^2(\Omega;{\mathcal H})$. We denote by
${\mathbb D}^{1,2}$ the closure of ${\mathcal S}$ with respect to
the norm
$$
\|F\|_{1,2}:=\sqrt{E|F|^2+E\|D^{H}F\|^2_{{\mathcal H}}}.
$$
The {\it divergence integral} $\delta^{H}$ is the adjoint of
derivative operator $D^{H}$. That is, we say that a random variable
$u$ in $L^2(\Omega;{\mathcal H})$ belongs to the domain of the
divergence operator $\delta^{H}$, denoted by ${\rm
{Dom}}(\delta^H)$, if
$$
E\left|\langle D^{H}F,u\rangle_{\mathcal H}\right|\leq
c\|F\|_{L^2(\Omega)}
$$
for every $F\in \mathcal S$. In this case $\delta^{H}(u)$ is defined
by the duality relationship
\begin{equation}\label{sec2-eq2.1}
E\left[F\delta^{H}(u)\right]=E\langle D^{H}F,u\rangle_{\mathcal H}
\end{equation}
for any $u\in {\mathbb D}^{1,2}$. We have ${\mathbb D}^{1,2}\subset
{\rm {Dom}}(\delta^H)$, and when $\frac12<H<1$ we have
\begin{align}\label{sec2-eq2.100}
E\left[\delta^{H}(u)^2\right]&=E\|u\|^2_{\mathcal
H}+E\int_{[0,T]^4}D^{H}_\xi u_rD^{H}_\eta
u_s\phi(\eta,r)\phi(\xi,s)dsdrd\xi d\eta
\end{align}
for any $u\in {\mathbb D}^{1,2}$. By the duality between $D^H$ and $\delta^H$ one have that the following result for the convergence of divergence integrals which is given in Naulart~\cite{Nua4}.
\begin{proposition}\label{prop2.1}
Let $\{u_n,n=1,2,\ldots\}\subset {\rm Dom}(\delta^H)$ such that $u_n\to u$
in $L^2(\Omega;{\mathbb H})$ for some $u\in L^2(\Omega;{\mathbb H})$. If that there exists $U\in L^2(\Omega)$ such that
$$
\delta^{H}(u_n)\longrightarrow U
$$
in $L^2(\Omega;{\mathbb H})$, as $n\to \infty$. Then, $u$ belongs to ${\rm Dom}(\delta^{H})$ and $\delta^{H}(u)=U$.
\end{proposition}
We will use the notation
$$
\delta^{H}(u)=\int_0^Tu_s\delta B^{H}_s
$$
to express the Skorohod integral of a process $u$, and the
indefinite Skorohod integral is defined as
$\int_0^tu_s\delta B^{H}_s=\delta^H(u{1}_{[0,t]})$. Recall the It\^o type formula for fBm $B^H$,
\begin{align*}
f(B_t^H)=f(0)+&\int_0^t
f'(B_s^H)\delta B_s^H+H\int_0^tf''(B_s^H)s^{2H-1}ds
\end{align*}
for any $f\in C^{2}({\mathbb R})$. Also recall
that $B^H$ has a local time ${\mathcal L}^{H}(x,t)$ continuous in
$(x,t)\in {\mathbb R}\times [0,\infty)$ which satisfies the
occupation formula (see Geman-Horowitz~\cite{Geman})
\begin{equation}\label{sec2-1-eq1}
\int_0^t\Phi(B_s^{H})ds=\int_{\mathbb R}\Phi(x){\mathcal L}^{H}(x,t)dx
\end{equation}
for every nonnegative bounded function $\Phi$ on ${\mathbb R}$, and such that
$$
{\mathcal L}^{H}(x,t)=\int_0^t\delta(B_s^H-x)ds=\lim_{\epsilon
\downarrow
0}\frac{1}{2\epsilon}\lambda\big(s\in[0,t],|B_s^H-x|<\epsilon\big),
$$
where $\lambda$ denotes Lebesgue measure and $\delta(x)$ is the
Dirac delta function. It is well-known that the local time ${\mathcal L}^H(x,t)$ has H\"older continuous paths of order $\gamma\in (0,1-H)$ in time, and of order $\kappa\in (0,\frac{1-H}{2H})$ in the space variable, provided $H\geq \frac13$. Define the so-called weighted local time
${\mathscr L}^{H}(x,t)$ of $B^H$ at $x$ as follows
$$
{\mathscr L}^{H}(x,t)=2H\int_0^ts^{2H-1}{\mathcal L}^{H}(x,ds)\equiv
2H\int_0^t\delta(B_s^H-x)s^{2H-1}ds.
$$
The H\"older continuity properties of ${\mathcal L}^H(x,t)$ can be transferred to the weighted local time ${\mathscr L}^{H}(x,t)$, and then ${\mathscr L}^{H}$ has a compact support in $x$, and the
following Tanaka formula holds (see Coutin {\it et al}~\cite{Cout} and Hu {\it et al}~\cite{Hu1}):
\begin{align}\label{eq2.4}
(B_t^H-x)^{-}&=(-x)^{-}-\int_0^t{1}_{\{B^H_s<
x\}}\delta B^H_s+\frac12{\mathscr L}^{H}(x,t).
\end{align}

At the end of this section we will establish some technical estimates associated with fractional Brownian motion. For simplicity we let $C$ stand for a positive constant depending only on the subscripts and its value may be different in different appearance, and this assumption is also adaptable to $c$.
\begin{lemma}\label{lem2.10}
For all $r,s\in [0,T],\;s\geq r$ and $0<H<1$ we have
\begin{equation}\label{eq3.1}
\frac12(2-2^H)r^{2H}(s-r)^{2H}\leq s^{2H}r^{2H}-\mu^2_{s,r}\leq
2r^{2H}(s-r)^{2H},
\end{equation}
where $\mu_{s,r}=E(B^H_sB^H_r)$.
\end{lemma}

By the local nondeterminacy of fBm we can prove the lemma (Yan {\em et al}~\cite{Yan1}), and Yan {\em et al}~\cite{Yan7} gave an elementary proof by using the inequality
\begin{align}\label{eq3.3000-1}
(1+x)^\alpha&\leq 1+(2^\alpha-1)x^\alpha
\end{align}
with $0\leq x,\alpha\leq 1$. It is important to note that inequality~\eqref{eq3.3000-1} is stronger than the well known
(Bernoulli) inequality
$$
(1+x)^\alpha\leq 1+\alpha x^\alpha\leq 1+x^\alpha,
$$
because $2^\alpha-1\leq \alpha$ for all $0\leq \alpha\leq 1$.
\begin{lemma}\label{lem2.20}
For all $s>r>0$ and $\frac12<H<1$ we have
\begin{equation}\label{eq2.1001}
c_H(s-r)rs^{2H-2}\leq \mu-r^{2H}\leq C_H(s-r)rs^{2H-2}
\end{equation}
and
\begin{equation}\label{eq2.1002}
c_H(s-r)s^{2H-1}\leq s^{2H}-\mu\leq C_H(s-r)s^{2H-1}
\end{equation}
where $\mu_{s,r}=E(B^H_sB^H_r)$.
\end{lemma}
\begin{proof}
For the inequalities~\eqref{eq2.1002} we have
$$
\mu-r^{2H}=\frac12\left(s^{2H}-r^{2H}-(s-r)^{2H}\right)
=\frac12s^{2H}\left(1-x^{2H}-(1-x)^{2H}\right)
$$
with $x=\frac{r}s$. By the continuity of the functions

$$
f_1(x)=\frac{1-x^{2H}-(1-x)^{2H}}{x(1-x)},\qquad f_2(x)=\frac{x(1-x)}{1-x^{2H}-(1-x)^{2H}}
$$
for $x\in (0,1)$ and $\lim_{x\to 0}f_i(x)=\lim_{x\to 1}f_i(x)=2H$ for $i=1,2$, we see that there exists a constant $C>0$ such that
$$
\frac1Cx(1-x)\leq 1-x^{2H}-(1-x)^{2H}\leq Cx(1-x)
$$
for all $x\in [0,1]$, which gives the inequalities~\eqref{eq2.1001}. The inequalities~\eqref{eq2.1002} is clear.
\end{proof}

\section{The existence of ${\mathcal C}^H(a)$}~\label{sec3}

Beside on the smooth approximation one can prove the existence of ${\mathcal C}^H$. In order to use the smooth approximation, we define the function $(x,y)\mapsto \Psi_{s,r,a,b}(x,y)$ on ${\mathbb R}^2$ by
\begin{align*}
\Psi_{s,r,a,b}(x,y):=\varphi_{s,r}&(x,y)-\varphi_{s,r}(x,b)\theta(1+b-y)\\
&-\varphi_{s,r}(a,y)\theta(1+a-x)
+\varphi_{s,r}(a,b)\theta(1+a-x)\theta(1+b-y)
\end{align*}
with $s,r>0$ and $a,b\in {\mathbb R}$, where $\theta(x)=1_{\{x>0\}}$ and $\varphi_{s,r}(x,y)$ denotes the density function of $(B^H_s,B^H_r)$. That is, \begin{equation}\label{sec3-eq3.2000}
\varphi_{s,r}(x,y)=\frac1{2\pi\rho_{s,r}}\exp\left\{
-\frac{1}{2\rho^2_{s,r}}\left( r^{2H}x^2-2\mu_{s,r}
xy+s^{2H}y^2\right)\right\},
\end{equation}
where $\mu_{s,r}=E(B^H_sB^H_r)$ and $\rho^2_{s,r}=(rs)^{2H}-\mu^2_{s,r}$.
Denote the density function of $B^H_s$ by $\varphi_s(x)$. The following Lemmas give some properties and estimates of $\Psi_{s,r,a,b}(x,y)$. The first lemma is a simple calculus exercise.
\begin{lemma}\label{lem3.-1}
Let $G_i\in C^\infty({\mathbb R})$ have compact supports for $i=1,2$. Then we have
\begin{equation}\label{sec3-eq3.5011}
\begin{split}
\int_{{\mathbb R}^2}&G'_1(x-a)G'_2(y-b)\varphi_{s,r}(x,y)dxdy\\
&=\int_{{\mathbb
R}^2} G'_1(x-a)G'_2(y-b)\Psi_{s,r,a,b}(x,y)dxdy\\
&\hspace{2cm}-G_2(1)\int_{{\mathbb R}}G_1(x-a)\frac{\partial}{\partial x}\varphi_{s,r}(x,b)dx\\
&\hspace{2cm}-G_1(1)\int_{{\mathbb R}}G_2(y-b)\frac{\partial}{\partial y}\varphi_{s,r}(a,y)dy-\varphi_{s,r}(a,b)G_1(1)G_2(1)
\end{split}
\end{equation}
for all $r,s>0$ and $a,b\in {\mathbb R}$, and moreover, if $G_i(1)=0$ for $i=1,2$, we then have
\begin{equation}\label{sec3-eq3.5000}
\int_{{\mathbb R}^2}G'_1(x-a)G'_2(y-b)\varphi_{s,r}(x,y)dxdy=\int_{{\mathbb
R}^2} G'_1(x-a)G'_2(y-b)\Psi_{s,r,a,b}(x,y)dxdy
\end{equation}
for all $r,s>0$ and $a,b\in {\mathbb R}$.
\end{lemma}

\begin{lemma}\label{lem3.0}
For any $x,y,z\in {\mathbb R}$ and $\beta\in [0,1]$ we have
\begin{equation}\label{lem3.0-eq1}
|\varphi_{s,r}(x,y)-\varphi_{s,r}(z,y)|\leq \frac{r^{\beta H}}{\rho_{s,r}^{1+\beta}}|x-z|^\beta e^{-\frac{\beta }{2r^{2H}}y^2}
\end{equation}
and
\begin{equation}\label{lem3.0-eq2}
|\varphi_{s,r}(x,y)-\varphi_{s,r}(x,z)|\leq \frac{s^{\beta H}}{\rho_{s,r}^{1+\beta}}|y-z|^\beta e^{-\frac{\beta }{2s^{2H}}x^2}.
\end{equation}
\end{lemma}
\begin{proof}
We have
\begin{align*}
\Bigl|&e^{-\frac{1}{2\rho^2_{s,r}}\left( r^{2H}x^2-2\mu_{s,r}
xy+s^{2H}y^2\right)}-e^{
-\frac{1}{2\rho^2_{s,r}}\left( r^{2H}z^2-2\mu_{s,r}
zy+s^{2H}y^2\right)}\Bigr|\\
&\qquad\qquad\leq \Bigl|e^{-\frac{1}{2\rho^2_{s,r}}\left( r^{2H}x^2-2\mu_{s,r}
xy+s^{2H}y^2\right)}-e^{
-\frac{1}{2\rho^2_{s,r}}\left( r^{2H}z^2-2\mu_{s,r}
zy+s^{2H}y^2\right)}\Bigr|^\beta
\end{align*}
for all $\beta\in [0,1]$. It follows from Mean Value Theorem that
\begin{align*}
|\varphi_{s,r}(x,y)&-\varphi_{s,r}(z,y)|\\
&\leq \frac{1}{2\pi\rho_{s,r}}\Bigl|e^{-\frac{1}{2\rho^2_{s,r}}\left( r^{2H}x^2-2\mu_{s,r}xy+s^{2H}y^2\right)}-e^{
-\frac{1}{2\rho^2_{s,r}}\left( r^{2H}z^2-2\mu_{s,r}
zy+s^{2H}y^2\right)}\Bigr|^\beta\\
&=\frac{1}{2\pi\rho_{s,r}}\Bigl|(x-z)(-\frac{1}{\rho^2_{s,r}}) (r^{2H}\xi-\mu_{s,r} y)e^{-\frac{1}{2\rho^2_{s,r}}\left(r^{2H}\xi^2-2\mu_{s,r}\xi y+s^{2H}y^2\right)}\Bigr|^\beta\\
&=\frac{1}{2\pi\rho_{s,r}}\Bigl|(x-z)\frac{r^{2H}}{\rho^2_{s,r}} (\xi-\frac{\mu_{s,r}}{r^{2H}}y)e^{-\frac{r^{2H}}{2\rho^2_{s,r}}( \xi-\frac{\mu_{s,r}}{r^{2H}}y)^2}e^{-\frac{y^2}{2r^{2H}}} \Bigr|^\beta
\end{align*}
for some $\xi$ between $z$ and $x$. Combining this with the fact $|x|e^{-x^2}\leq 1$, we get
\begin{align*}
|\varphi_{s,r}(x,y)&-\varphi_{s,r}(z,y)|\leq \frac{r^{\beta H}}{\rho_{s,r}^{1+\beta}}|x-z|^\beta e^{-\frac{\beta }{2r^{2H}}y^2}
\end{align*}
for all $\beta\in [0,1]$. Similarly, one can obtain the estimate~\eqref{lem3.0-eq2}.
\end{proof}

\begin{lemma}\label{lem3.1}
The estimate
\begin{align*}
\Lambda_1(s,r,a,b):=\int_a^{\infty}\int_b^\infty
\frac{|\Psi_{s,r,a,b}(x,y)|}{(x-a)(y-b)}dxdy\leq
\frac{C_{H,T,\beta} s^{\beta H/2}}{r^{(1+\beta)H}(s-r)^{(1+\beta)H}}
\end{align*}
holds for all $\beta\in (0,1),0<r<s\leq T$ and $a,b\in {\mathbb R}$.
\end{lemma}
\begin{proof}
We have
\begin{align}\notag
\Lambda_1(s,r,a,b)&=\int_a^\infty\int_b^\infty \frac{1}{(x-a)(y-b)}
|\Psi_{s,r,a,b}(x,y)|dxdy\\  \label{eq3-2020110}
&\hspace{-1cm}\leq \int_a^{a+1}dx\int_b^{b+1}\frac{1}{(x-a)(y-b)}| \varphi_{s,r}(x,y)
-\varphi_{s,r}(x,b)-\varphi_{s,r}(a,y) +\varphi_{s,r}(a,b)|dy\\ \notag
&\qquad+\int_{a+1}^\infty dx\int_b^{b+1}\frac{1}{(x-a)(y-b)}| \varphi_{s,r}(x,y)-\varphi_{s,r}(x,b)|dy\\ \notag
&\qquad+\int_a^{a+1}dx\int_{b+1}^\infty \frac{1}{(x-a)(y-b)}|\varphi_{s,r}(x,y)
-\varphi_{s,r}(a,y)|dy\\  \notag
&\qquad+\int_{a+1}^\infty dx\int_{b+1}^\infty \frac{1}{(x-a)(y-b)}\varphi_{s,r}(x,y)dy\\ \notag
&\hspace{-1cm}\equiv \Lambda_{11}(s,r,a,b)+\Lambda_{12}(s,r,a,b)
+\Lambda_{13}(s,r,a,b)+\Lambda_{14}(s,r,a,b).
\end{align}
Clearly, $\Lambda_{14}(s,r,a,b)\leq 1$ and we have
\begin{align*}
\Lambda_{12}(s,r,a,b)+\Lambda_{13}(s,r,a,b)&\leq 2\frac{s^{\beta H}}{\rho^{1+\beta}_{s,r}}\int_{a+1}^\infty dx\int_b^{b+1}\frac{1}{ (x-a)(y-b)^{1-\beta}} e^{-\frac{\beta}{2s^{2H}}x^2}dy\\
&\qquad+2\frac{r^{\beta H}}{\rho^{1+\beta}_{s,r}}\int_a^{a+1}dx
\int_{b+1}^\infty \frac{1}{(x-a)^{1-\beta}(y-b)}
e^{-\frac{\beta}{2r^{2H}}y^2}dy\\
&\leq \frac{C_{H,\beta}s^{(1+\beta)H}}{r^{(1+\beta)H}(s-r)^{(1+\beta)H}}
\end{align*}
and $a,b\in {\mathbb R}$ by Lemma~\ref{lem3.0} with $\beta\in (0,1)$ and Lemma~\ref{lem2.10}. In order to estimate $\Lambda_{11}(s,r,a,b)$, by Lemma~\ref{lem3.0} we have
\begin{align}\label{lem3.1-eq1-111}
|\varphi_{s,r}(x,y)
&-\varphi_{s,r}(x,b)-\varphi_{s,r}(a,y) +\varphi_{s,r}(a,b)| \leq 2\frac{s^{\beta H}}{\rho_{s,r}^{1+\beta}}|y-b|^\beta
\end{align}
and
\begin{equation}\label{sec3-eq3.8111}
\begin{split}
|\varphi_{s,r}(x,y)
&-\varphi_{s,r}(x,b)-\varphi_{s,r}(a,y) +\varphi_{s,r}(a,b)| \leq 2\frac{r^{\beta H}}{\rho_{s,r}^{1+\beta}}|x-a|^\beta
\end{split}
\end{equation}
for all $a,b,x,y\in {\mathbb R}$, which give
\begin{equation}\label{lem3.1-eq1}
|\varphi_{s,r}(x,y)
-\varphi_{s,r}(x,b)-\varphi_{s,r}(a,y) +\varphi_{s,r}(a,b)|\leq 2\frac{(sr)^{\beta H/2}}{\rho_{s,r}^{1+\beta}}|(x-a)(y-b)|^{\beta/2}.
\end{equation}
It follows from Lemma~\ref{lem2.10} that
\begin{align*}
\Lambda_{11}(s,r,a,b)&=\int_a^{a+1}\!\!dx\int_b^{b+1}\!\!\frac{dy}{(x-a)(y-b)} |\varphi_{s,r}(x,y)-\varphi_{s,r}(x,b)-\varphi_{s,r}(a,y) +\varphi_{s,r}(a,b)|\\
&\leq 2\frac{(sr)^{\beta H/2}}{\rho_{s,r}^{1+\beta}}\int_a^{a+1}dx\int_b^{b+1} \frac{1}{(x-a)^{1-\frac\beta2} (y-b)^{1-\frac\beta2} }dy\\
&\leq C_{H,\beta}\frac{s^{\beta H/2}}{r^{(2+\beta)H/2}(s-r)^{(1+\beta)H}}
\end{align*}
for all $\beta\in (0,1)$ and $a,b\in {\mathbb R}$. This completes the proof.
\end{proof}

The next proposition shows the process
\begin{equation}\label{sec4-eq4.1}
{\mathcal C}^{+,H}_t(a):=2\left( F_{+}(B^H_t-a)-F_{+}(-a)-\int_0^tF_{+}'(B^H_s-a)
\delta B^H_s\right)
\end{equation}
exists in $L^2(\Omega)$, where
\begin{equation}\label{sec3-eq3.4}
F_{+}(x)=
\begin{cases}
0,& {\text {if $x\leq 0$}},\\
x\log{x}-x,& {\text {if $x>0$}}.
\end{cases}
\end{equation}

\begin{proposition}\label{lem3.2}
Let the function $F_{+}$ be given as above. Then the random variable
\begin{equation}\label{sec3-eq3.5}
\int_0^tF_{+}'(B^H_s-a)\delta B^H_s
\end{equation}
exists and
\begin{equation}\label{sec3-eq3.802}
\begin{split}
E\left|\int_0^tF_{+}'(B^H_s-a)\delta B^H_s\right|^2&
=\int_0^t\int_0^tE[F'_{+}(B^H_s-a)F'_{+}(B^H_r-a)]\phi(s,r)dsdr\\
&\hspace{-2.5cm}+\int_0^tds\int_0^sd\xi\int_0^tdr\int_0^rd\eta
\phi(s,\eta)\phi(r,\xi)\int_a^{\infty}\int_a^\infty
\frac{\Psi_{s,r,a,a}(x,y)dxdy}{(x-a)(y-a)}
\end{split}
\end{equation}
for all $t\geq 0$ and $a\in {\mathbb R}$.
\end{proposition}

By smooth approximation we can obtain the statement. Define the function $\zeta$ on ${\mathbb R}$ by
\begin{equation}\label{Ito-formula-1}
\zeta(x):=
\begin{cases}
ce^{\frac1{(x-1)^2-1}}, &{\text {$x\in (0,2)$}},\\
0, &{\text {otherwise}},
\end{cases}
\end{equation}
where $c$ is a normalizing constant such that $\int_{\mathbb
R}\zeta(x)dx=1$. Define the mollifiers
\begin{equation}\label{Ito-formula-101}
\zeta_n(x):=n\zeta(nx),\qquad n=2,\ldots
\end{equation}
and the sequence of smooth functions
\begin{align}\label{eq3-2020202}
G_{n}(x):&=\int_{\mathbb
R}F'_{+}(y)\zeta_n(x-y)dy=n\int_{x-\frac{2}n}^x
F'_{+}(y)\zeta(n(x-y))dy\\  \label{eq3-2020203}
&=\int_0^2F'_{+}(x-\frac{y}n)\zeta(y)dy,\qquad n=2,\ldots.
\end{align}
Then $G_{n}\in C^\infty({\mathbb R})$ with compact support for all $n\geq 2$.

\begin{lemma}\label{lem3.301}
Let the functions $G_n,n\geq 2$ be defined as above. Then we have
$$
|G_{n}(x)|\leq \psi_1(x):=
\begin{cases}
0, & {\text { if $x\leq 0$}},\\
C(1+|\log x|), & {\text { if $x>0$}}
\end{cases}
$$
for all $x\in {\mathbb R}$, and
\begin{equation}\label{eq3-20001}
G_{n}(x)\longrightarrow F'_{+}(x)
\end{equation}
for all $x\neq 0$, as $n$ tends to infinity.
\end{lemma}
\begin{proof}
Clearly, $G_{n}(x)=0$ for $x\leq 0$ and
\begin{equation}\label{eq3-0012}
G_{n}(x)=n\int_{(x-\frac{2}n)\vee 0}^x\zeta(n(x-y))\log ydy
\end{equation}
for $x>0$.

When $0<x\leq \frac2n$, we have
\begin{align*}
|G_{n}(x)|&\leq n\int_0^x\zeta(n(x-y))|\log y|dy\leq -n\int_0^x\log ydy\\
&=nx(1-\log x)\leq 2(1-\log x).
\end{align*}
On the other hand, by~\eqref{eq3-2020203} we get
\begin{align*}
G_{n}(x)&=n\int_0^2\zeta(z)\log(x-\frac{z}{n})dy =\int_0^{2}\zeta(z)\log[x(1-\frac{z}{nx})]dz\\
&=\log x\int_0^2\zeta(z)dz+\int_0^2\zeta(z)\log(1-\frac{z}{nx})dz\\
&=\log x+\int_0^2\zeta(z)\log(1-\frac{z}{nx})dz
\end{align*}
for $x>\frac{2}n$, which gives
\begin{align*}
|G_{n}(x)|\leq C\left(1+|\log x| \right)
\end{align*}
for $x>\frac{2}n$ since
\begin{align*}
\int_0^2|\log(1-\frac{z}{u})|e^{-\frac1{1-(1-z)^2}}dz<\infty
\end{align*}
with $u>2$.

Finally, the convergence~\eqref{eq3-20001} follows from $G_{n}(x)=F'_\alpha(x)=0$ for $x<0$ and the next estimate:
\begin{align*}
|G_{n}(x)-F'_{+}(x)|&\leq  \alpha^{-1}\int_0^2\zeta(y)\left|F'_{+}(x-\frac{y}n)-F'_{+}(x)\right|dy\\
&=\alpha^{-1}\int_0^2\left|\log(x-\frac{y}n)-\log x\right|\zeta(y)dy\\
&\leq \alpha^{-1}\int_0^2 \log\left(1+\frac{y}{nx-y}\right)\zeta(y)dy\\
&\leq \frac{1}{\alpha}\log\left(1+\frac{2}{nx-2}\right)\int_0^2\zeta(y)dy =\frac{1}{\alpha}\log\left(1+\frac{2}{nx-2}\right)
\end{align*}
for all $x>\frac2n$. This completes the proof.
\end{proof}

\begin{lemma}\label{lem3.301-1}
Let the functions $G_n,n\geq 2$ be defined as above. Then we have
$$
|G'_{n}(x)|\leq \psi_2(x):=
\begin{cases}
0, & {\text { if $x\leq 0$}},\\
Cx^{-1}(1+|\log x|), & {\text { if $x>0$}}
\end{cases}
$$
for any $x\in {\mathbb R}$, and
\begin{equation}\label{eq3-20002}
G'_{n}(x)\longrightarrow F''_{+}(x)
\end{equation}
for all $x\neq 0$, as $n$ tends to infinity.
\end{lemma}
\begin{proof}
Clearly, $G'_{n}(x)=0$ for $x\leq 0$, and we have for $x>\frac{2}n$,
\begin{equation}\label{eq3-200021}
\begin{split}
G'_{n}(x)&=\int_0^2 F''_{+}(x-\frac{y}n)\zeta(y)dy= \int_0^2\zeta(y)\frac1{x-\frac{y}n}dy\\
&=\int_0^2\zeta(y)\frac{n}{nx-y} dy=
\frac1x\int_0^2\zeta(y)\left(1+\frac{y}{nx-y}\right)dy
\end{split}
\end{equation}
by~\eqref{eq3-2020202}. It follows that
\begin{align*}
|G'_{n}(x)|&\leq
\frac1x\int_0^2\zeta(y)\left(1+\frac{y}{2-y}\right)dy\leq \frac{C}x
\end{align*}
for $x>\frac{2}n$. On the other hand, for $0<x\leq \frac2n$ we have
\begin{align*}
G'_{n}(x)&=n\int_{\mathbb R}F'_{+}(y)\frac{\partial}{\partial x}\zeta(n(x-y))dy\\
&=-2n^2\int_{x-\frac{2}n}^x\frac{1-n(x-y)}{(1-(1-n(x-y))^2)^2}\zeta(n(x-y)) F'_{+}(y)dy.
\end{align*}
Combining this with the fact
$$
x^2e^{-x}\leq 2\quad (x\geq 0)
$$
lead to
\begin{align*}
|G'_{n}(x)|&\leq 4n^2\int_0^x|F'_{+}(y)||1-n(x-y)|dy\leq 8n^2\int_0^x|F'_{+}(y)|dy\\
&=-8n^2\int_0^x\log ydy=8n^2x(1-\log x)\leq \frac{32}{x}(1-\log x)
\end{align*}
for $0<x\leq \frac2n$, which gives the estimates of $G'_{n}(x)$ with $0<x\leq \frac2n$.

Finally, by the estimate
$$
\int_0^2\zeta(y)\left(1+\frac{y}{nx-y}\right)dy\leq \int_0^2\zeta(y)\left(1+\frac{y}{2-y}\right)dy<\infty
$$
for all $x>\frac2n$ and Lebesgue's dominated convergence theorem we have
$$
\lim_{u\to \infty}\int_0^2\zeta(y)\left(1+\frac{y}{u-y}\right)dy=1.
$$
Combining this with~\eqref{eq3-200021}, we get the convergence~\eqref{eq3-20002} since $G'_{n}(x)=F''_{+}(x)=0$ for all $x<0$. This completes the proof.
\end{proof}

\begin{lemma}\label{lem3.301-3}
Let $\psi_2$ be defined in Lemma~\ref{lem3.301-1}. Then the estimate
\begin{align*}
\int_a^{\infty}\int_b^\infty
\psi_2(x-a)\psi_2(y-b)|\Psi_{s,r,a,b}(x,y)|dxdy\leq
\frac{C_{H,t}s^{\gamma H/2}}{r^{(1+\gamma)H}(s-r)^{(1+\gamma)H}}
\end{align*}
holds for all $\gamma\in (0,1),0<r<s\leq t$ and $a,b\in {\mathbb R}$.
\end{lemma}
\begin{proof}
Similar to Lemma~\ref{lem3.1} one can obtain the estimate since
$$
|\log x|\leq C(x^{-\beta}+x^\beta)
$$
for all $x>0$ and all $0<\beta<1$.
\end{proof}

\begin{lemma}\label{lem3.301-4}
Let $\psi_1$ be defined in Lemma~\ref{lem3.301}. Then we have
\begin{align}\label{sec3-eq3.21007}
&\int_{{\mathbb R}}\psi_1(x-a)\left|\frac{\partial}{\partial x}\varphi_{s,r}(x,a)\right|dx\leq C_{H,t,\alpha}(s-r)^{-(1+\alpha)H}r^{-(1+\alpha)H}\\ \label{sec3-eq3.22001007}
&\int_{{\mathbb R}}\psi_1(y-a)\left|\frac{\partial}{\partial y}\varphi_{s,r}(a,y)\right|dy \leq C_{H,t,\alpha}(s-r)^{-(1+\alpha)H}r^{-(1+2\alpha)H}
\end{align}
for all $a\in {\mathbb R}$, $0<r<s\leq t$ and $1-H<\alpha<1$.
\end{lemma}
\begin{proof}
Given $a\in {\mathbb R}$ and $0<r<s\leq t$. Make the substitution $\frac{r^H}{\rho}(x-\frac{\mu}{r^{2H}}a)=y$. Then
\begin{align*}
\int_a^\infty&|\log(x-a)|\left|\frac{\partial}{\partial x}\varphi_{s,r}(x,a)\right|dx\\
&=\frac{r^{H}}{\rho}\int_a^\infty|\log(x-a)| \left|\frac{r^{H}}{\rho}\left(x-\frac{\mu}{r^{2H}}a\right)\right|\varphi_{s,r}(x,a)dx\\
&=\frac1{2\pi\rho}e^{-\frac{a^2}{2r^{2H}}}\int_{-\frac{\mu-r^{2H}}{\rho r^H}a}^\infty\left|\log\left\{\frac{\rho}{r^H}\left(y+\frac{\mu-r^{2H}}{\rho r^{H}}a \right)\right\}\right||y|e^{-\frac12y^2}dy\\
&\leq \frac1{2\pi\rho}\left(\int_{-\frac{\mu-r^{2H}}{\rho r^H}a}^\infty\left|\log\frac{\rho}{r^H}\right||y|e^{-\frac12y^2}dy\right.\\
&\qquad\left.+
e^{-\frac{a^2}{2r^{2H}}}\int_{-\frac{\mu-r^{2H}}{\rho r^H}a}^\infty\left|\log\left(y+\frac{\mu-r^{2H}}{\rho r^{H}}a \right)\right||y|e^{-\frac12y^2}dy\right)\\
&\equiv \frac1{\rho}\left(\Delta_1(s,r,a)+\Delta_2(s,r,a)\right).
\end{align*}
By Lemma~\ref{lem2.10} and the fact $|\log x|\leq x+x^{-\alpha}$ for all $x>0$ and $\alpha\in (0,1)$ we see that
\begin{align*}
\Delta_1(s,r,a)\leq \left(\left(\frac{\rho}{r^H}\right)^{-\alpha}+\frac{\rho}{r^H}
\right)\int_{-\frac{\mu-r^{2H}}{\rho r^H}a}^\infty |y|e^{-\frac12y^2}dy
\leq \frac{r^{H\alpha}}{\rho^\alpha}+\frac{\rho}{r^H}\leq \frac{C_{H,t,\alpha}}{(s-r)^{H\alpha}}
\end{align*}
and

\begin{align*}
\Delta_2(s,r,a)&\leq e^{-\frac{a^2}{2r^{2H}}}\int_{-\frac{\mu-r^{2H}}{\rho r^H}a}^\infty\left(y+\frac{\mu-r^{2H}}{\rho r^{H}}a \right)|y|e^{-\frac12y^2}dy\\
&\qquad+e^{-\frac{a^2}{2r^{2H}}}\int_{-\frac{\mu-r^{2H}}{\rho r^H}a}^\infty\left(y+\frac{\mu-r^{2H}}{\rho r^{H}}a \right)^{\alpha-1}|y|e^{-\frac12y^2}dy\\
&\equiv \Delta_{21}(s,r,a)+\Delta_{22}(s,r,a).
\end{align*}
Now, let us estimate $\Delta_{21}(s,r,a)$ and $\Delta_{22}(s,r,a)$. We have
\begin{align*}
\Delta_{21}(s,r,a)
&\leq \int_{-\frac{\mu-r^{2H}}{\rho r^H}a}^\infty |y|^2e^{-\frac12y^2}dy+e^{-\frac{a^2}{2r^{2H}}}\int_{-\frac{\mu-r^{2H}}{\rho r^H}a}^\infty\frac{\mu-r^{2H}}{\rho r^{H}}|a||y|e^{-\frac12y^2}dy\\
&\leq \int_{-\infty}^\infty|y|^2e^{-\frac12y^2}dy +\left(\frac{|a|}{r^H}e^{-\frac{a^2}{2r^{2H}}}\right)\left( \frac{\mu-r^{2H}}{\rho}\right)\int_{-\infty}^\infty|y|e^{-\frac12y^2}dy\\
&\leq \sqrt{2\pi}\left(1+\frac{\mu-r^{2H}}{\rho}\right)
\end{align*}
by the fact $|y|e^{-\frac12y^2}\leq 1$. On the other hand, we have also
\begin{align*}
\Delta_{22}&(s,r,a)1_{\{a\geq 0\}}\\
&\leq e^{-\frac{a^2}{2r^{2H}}}\left(\int_{-\frac{\mu-r^{2H}}{\rho r^H}a}^0\left(y+\frac{\mu-r^{2H}}{\rho r^{H}}a \right)^{\alpha-1}|y|e^{-\frac12y^2}dy+\int_0^\infty y^{\alpha+1}e^{-\frac12y^2}dy\right)1_{\{a\geq 0\}}\\
&\leq \frac1{\alpha}e^{-\frac{a^2}{2r^{2H}}}\left(\frac{\mu-r^{2H}}{\rho r^H}a\right)^{\alpha}1_{\{a\geq 0\}}+C_\alpha1_{\{a\geq 0\}}\\
&\leq C_\alpha\frac{(\mu-r^{2H})^\alpha}{\rho^{\alpha}}1_{\{a\geq 0\}}+C_\alpha1_{\{a\geq 0\}}
\end{align*}
and
\begin{align*}
\Delta_{22}(s,r,a)1_{\{a<0\}}&\leq
1_{\{a<0\}}\int_{-\frac{\mu-r^{2H}}{\rho r^H}a}^{\frac{\mu-r^{2H}}{\rho r^H}(1-a)}\left(y+\frac{\mu-r^{2H}}{\rho r^{H}}a\right)^{\alpha-1}|y|e^{-\frac12y^2}dydy\\
&\qquad+1_{\{a<0\}}\int_{\frac{\mu-r^{2H}}{\rho r^H}(1-a)}^\infty\left(\frac{\mu-r^{2H}}{\rho r^{H}}\right)^{\alpha-1}|y|e^{-\frac12y^2}dydy\\
&\leq \alpha^{-1}\frac{(\mu-r^{2H})^{\alpha}}{(\rho r^{H})^{\alpha}}1_{\{a<0\}} +\frac{(\mu-r^{2H})^{\alpha-1}}{(\rho r^{H})^{\alpha-1}}1_{\{a<0\}}
\end{align*}
by the fact $|y|^\alpha e^{-\frac12y^2}\leq 1$ with $0\leq \alpha\leq 1$. It follows from Lemma~\ref{lem2.10} and Lemma~\ref{lem2.20} that
\begin{align*}
\Delta_2(s,r,a)&=\Delta_{21}(s,r,a)+\Delta_{22}(s,r,a)\\
&\leq C_\alpha\left(1+\frac{\mu-r^{2H}}{\rho}+
\frac{(\mu-r^{2H})^\alpha}{\rho^{\alpha}}+\frac{(\mu-r^{2H})^\alpha}{(\rho r^{H})^\alpha} +\frac{(\mu-r^{2H})^{\alpha-1}}{(\rho r^{H})^{\alpha-1}}\right)\\
&\leq C_{H,\alpha,t}(s-r)^{-(1-H)(1-\alpha)}r^{-\alpha H}
\end{align*}
for all $0<r<s\leq t$. Combining this with Lemma~\ref{lem2.10}, we have
\begin{align*}
\int_{{\mathbb R}}&\psi_1(x-a)\left|\frac{\partial}{\partial x}\varphi_{s,r}(x,a)\right|dx\\
&=\int_{{\mathbb R}}\left|\frac{\partial}{\partial x}\varphi_{s,r}(x,a)\right|dx+\int_a^\infty |\log(x-a)|\left|\frac{\partial}{\partial x}\varphi_{s,r}(x,a)\right|dx\\
&\leq \frac{C_{H,\alpha,t}}\rho\left(1+\Delta_1(s,r,a)+\Delta_2(s,r,a)\right)\\
&\leq \frac{C_{H,\alpha,t}}{\rho}\left(1+
(s-r)^{-H\alpha}+r^{-\alpha H}(s-r)^{-(1-H)(1-\alpha)}\right)\\
&\leq C_{H,\alpha,t}(s-r)^{-(1+\alpha)H}r^{-(1+\alpha)H}
\end{align*}
for all $0<r<s\leq t$ and $1-H<\alpha\leq 1$. Similarly, we can obtain the estimate~\eqref{sec3-eq3.22001007}.
\end{proof}

Now, we can prove Proposition~\ref{lem3.2}.
\begin{proof}[Proof of Proposition~\ref{lem3.2}]
Let $G_n,n\geq 2$ be defined in~\eqref{eq3-2020202}. Then
$$
E\left|G_{n}(B^H_s-a)-F'_{+}(B^H_s-a)\right|^2\longrightarrow 0\quad (n\to \infty)
$$
for all $s\geq 0$ and $a\in {\mathbb R}$ by Lemma~\ref{lem3.301}, Lebesgue's dominated convergence theorem and the next estimate:
\begin{equation}\label{prop3.1-eq1001}
\begin{split}
E[\psi_1(B^H_s-a)^2]&=\int_a^\infty (1+|\log(x-a)|)^2\varphi_{s}(x)dx<\infty
\end{split}
\end{equation}
for all $s\geq 0$ and $a\in {\mathbb R}$. Thus, it is sufficient to show that
the sequence
$$
Y^H_t(n):=\int_0^tG_{n}(B^H_s-a)\delta B^H_s,\quad n\geq 2
$$
is a Cauchy sequence in $L^2(\Omega)$. Denote $\widetilde{G}_{n,m}=G_{n}-G_{m}$ for all $n,m\geq 2$. Then $Y^H_t(n)$ is a Cauchy sequence in $L^2(\Omega)$ if and only if
\begin{align*}
E\left|Y^H_t(n)-Y^H_t(m)\right|^2&=E\left|\int_0^t\widetilde{G}_{n,m}(B^H_s-a) \delta B^H_s\right|^2\\
&=\int_0^t\int_0^tE\widetilde{G}_{n,m}(B^H_s-a) \widetilde{G}_{n,m} (B^H_r-a)\phi(s,r)drds\\
&\qquad+\int_0^tds\int_0^sd\xi\int_0^tdr\int_0^rd\eta
\phi(s,\eta)\phi(r,\xi)\\
&\hspace{4cm}\cdot E\left[\widetilde{G}_{n,m}'(B^H_s-a) \widetilde{G}_{n,m}'(B^H_r-a)\right]\\
&\equiv \Lambda_{n,m}(1)+\Lambda_{n,m}(2)\longrightarrow 0
\end{align*}
as $n,m\to \infty$.

On the one hand, we have
$$
|\widetilde{G}_{n,m}(x)|\leq C\psi_1(x)\quad (x\in {\mathbb R})
$$
and $\widetilde{G}_{n,m}(x)\to 0$ for all $x\in {\mathbb R}$, as $n,m$ tends to infinity, by Lemma~\ref{lem3.301}. Accrediting with the estimate
\begin{equation}\label{prop3.1-eq1}
\begin{split}
\Lambda_{2}(s,r,a,a)&:=\int_0^t\int_0^t E[\psi_1(B^H_s-a)\psi_1(B^H_r-a)]\phi(s,r)dsdr\\
&=\int_0^t\int_0^t\phi(s,r)dsdr
\int_a^\infty\int_a^\infty (1+|\log(x-a)|)\\
&\qquad\qquad\qquad\cdot(1+|\log(y-a)|)\varphi_{s,r}(x,y)dxdy<\infty
\end{split}
\end{equation}
and Lebesgue's dominated convergence theorem, we give the convergence
\begin{equation}\label{sec3-eq3.80200011=}
\Lambda_{n,m}(1)=\int_0^t\int_0^t\phi(s,r)drds\int_{{\mathbb R}^2}\widetilde{G}_{n,m}(x-a)\widetilde{G}_{n,m}(y-a) \varphi_{s,r}(x,y)dxdy \longrightarrow 0
\end{equation}
for $a\in {\mathbb R}$, as $n,m$ tend to infinity.

On the other hand, by Lemma~\ref{lem3.-1} we have
\begin{equation}\label{sec3-eq3.802000134}
\begin{split}
\Lambda_{n,m}(2)&=\int_0^tds\int_0^sd\xi\int_0^tdr\int_0^rd\eta
\phi(s,\eta)\phi(r,\xi)\\
&\qquad\qquad \cdot\int_{\mathbb{R}^2}\widetilde{G}'_{n,m}(x-a)\widetilde{G}'_{n,m}(y-a)
\Psi_{s,r,a,a}(x,y)dxdy\\
&\quad+\int_0^tds\int_0^tdr\int_0^sd\xi\int_0^rd\eta \phi(s,\eta)\phi(r,\xi)
\Theta_{n,m}(s,r,a,a)
\end{split}
\end{equation}
for all $n,m$, $t\geq 0$ and $a\in {\mathbb R}$, where
\begin{align*}
\Theta_{n,m}(s,r,a,b)&=-\widetilde{G}_{n,m}(1)\int_{{\mathbb R}}\widetilde{G}_{n,m}(x-a)\frac{\partial}{\partial x}\varphi_{s,r}(x,b)dx\\
&\quad-\widetilde{G}_{n,m}(1)\int_{{\mathbb R}}\widetilde{G}_{n,m}(y-b)\frac{\partial}{\partial y}\varphi_{s,r}(a,y)dy-\varphi_{s,r}(a,b)\left(\widetilde{G}_{n,m}(1)\right)^2.
\end{align*}
Noting that
\begin{align*}
\int_0^r|s-\xi|^{2H-2}d\xi=\frac1{2H-1}\left(s^{2H-1}
+|s-r|^{2H-1}{\rm sign}(r-s)\right),
\end{align*}
we get
\begin{equation}\label{eq3-300}
\int_0^sd\xi\int_0^r|r-\xi|^{2H-2}|s-\eta|^{2H-2}d\eta \leq
\frac2{(2H-1)^2}r^{2H-1}s^{2H-1}.
\end{equation}
It follows from Lemma~\ref{lem3.301} and Lemma~\ref{lem3.301-4} with $1-H<\alpha<\frac{1-H}H\wedge\frac12$ that
\begin{equation}\label{sec3-eq3.802000188}
\begin{split}
\int_0^tds&\int_0^tdr\int_0^sd\xi\int_0^rd\eta \phi(s,\eta)\phi(r,\xi)
|\Theta_{n,m}(s,r,a,a)|\\
&\leq C_{H,t}\widetilde{G}_{n,m}(1)\int_0^tds\int_0^tdr\int_0^sd\xi\int_0^rd\eta \frac{\phi(s,\eta)\phi(r,\xi)}{|s-r|^{H(1+\alpha)}(s\wedge r)^{(1+\alpha)H}}\\
&\qquad+C_{H,t} \widetilde{G}_{n,m}(1)\int_0^tds\int_0^tdr\int_0^sd\xi\int_0^rd\eta \frac{\phi(s,\eta)\phi(r,\xi)}{|s-r|^{H(1+\alpha)}(s\wedge r)^{(1+2\alpha)H}}\\
&\qquad+\left(\widetilde{G}_{n,m}(1)\right)^2 \int_0^tds\int_0^tdr\int_0^sd\xi\int_0^rd\eta \phi(s,\eta)\phi(r,\xi)\frac1{\rho}\\
&\leq C_{H,t}\widetilde{G}_{n,m}(1)\left(1+\widetilde{G}_{n,m}(1)\right) \longrightarrow 0\quad(n,m\to \infty)
\end{split}
\end{equation}
for all $t\geq 0$ and $a\in {\mathbb R}$. Moreover,~\eqref{eq3-300} and Lemma~\ref{lem3.301-3} imply that
\begin{equation}\label{sec3-eq3.802000108}
\begin{split}
\int_0^tds\int_0^sd\xi\int_0^tdr&\int_0^rd\eta
\phi(s,\eta)\phi(r,\xi)\\
&\cdot\int_{\mathbb{R}^2}\psi_2(x-a)\psi_2(y-a)
\Psi_{s,r,a,a}(x,y)dxdy<\infty
\end{split}
\end{equation}
for all $t\geq 0$ and $a\in {\mathbb R}$. Combining this with~\eqref{sec3-eq3.802000134},~\eqref{sec3-eq3.802000188},  Lemma~\ref{lem3.301-1} and Lebesgue's dominated convergence theorem that the convergence, we have
$$
\Lambda_{n,m}(2)\longrightarrow 0,
$$
as $n,m$ tend to infinity. Thus, we have showed that $Y^H_t(n),n=1,2,\ldots$ is a Cauchy sequence in $L^2(\Omega)$ and the process
$$
\lim_{n\to \infty}\int_0^tG_n(B^H_s-a)\delta B^H_s=\int_0^tF'_{+}(B^H_s-a)\delta B^H_s,\quad t\geq 0
$$
exists in $L^2(\Omega)$.

Denote
\begin{align*}
\widetilde{\Theta}_{n}(s,r,a,b)&:=-G_{n}(1)\int_{{\mathbb R}}G_{n}(x-a)\frac{\partial}{\partial x}\varphi_{s,r}(x,b)dx\\
&\quad-G_{n}(1)\int_{{\mathbb R}}G_{n}(y-b)\frac{\partial}{\partial y}\varphi_{s,r}(a,y)dy-\varphi_{s,r}(a,b)G_{n}(1)G_{n}(1)
\end{align*}
for all $a,b\in {\mathbb R}$ and $0<r<s$. Then, for all $t\geq 0$ and $a\in {\mathbb R}$ we have
\begin{align*}
E\left|\int_0^tG_{n}(B^H_s-a)\delta B^H_s\right|^2 &=\int_0^t\int_0^tEG_{n}(B^H_s-a)G_{n}(B^H_r-a)\phi(s,r)drds\\
&\qquad+\int_0^tds\int_0^sd\xi\int_0^tdr\int_0^rd\eta
\phi(s,\eta)\phi(r,\xi)\\
&\hspace{2.5cm}\cdot E\left[G'_{n}(B^H_s-a)G'_{n} (B^H_r-a)\right]\\
&=\int_0^t\int_0^tEG_{n}(B^H_s-a)G_{n}(B^H_r-a)\phi(s,r)drds\\
&\qquad+\int_0^tds\int_0^sd\xi\int_0^tdr\int_0^rd\eta
\phi(s,\eta)\phi(r,\xi)\\
&\hspace{2.5cm}\cdot \int_{{\mathbb R}^2}  G'_{n}(x-a)G'_{n}(y-a)\Psi_{s,r,a,a}(x,y)dxdy\\
&\qquad+\int_0^tds\int_0^tdr\int_0^sd\xi\int_0^rd\eta \phi(s,\eta)\phi(r,\xi)\widetilde{\Theta}_{n}(s,r,a,a)
\end{align*}
by Lemma~\ref{lem3.0}. Notice that
$$
\int_0^tds\int_0^tdr\int_0^sd\xi\int_0^rd\eta \phi(s,\eta)\phi(r,\xi)\widetilde{\Theta}_{n}(s,r,a,a)\longrightarrow 0,
$$
as $n$ tends to infinity, by Lemma~\ref{lem3.301}, Lemma~\ref{lem3.301-4} and~\eqref{eq3-300}. We introduce the identity~\eqref{sec3-eq3.802} by taking the limit in $L^2(\Omega)$ and the proposition follows.
\end{proof}
Finally, by considering the function on ${\mathbb R}^2$
\begin{align*}
\widetilde{\Psi}_{s,r,a,b}(x,y):=\varphi_{s,r}(x,y)&-\varphi_{s,r}(x,b)\theta(y-1-b)\\
&-\varphi_{s,r}(a,y)\theta(x-1-a)
+\varphi_{s,r}(a,b)\theta(x-1-a)\theta(y-1-b)
\end{align*}
with $s,r>0$ and $a,b\in {\mathbb R}$, and in a same way as proof of Proposition~\ref{lem3.2}, we can show that the integral
$$
\int_0^tF_{-}'(B^H_s-a)\delta B^H_s,\quad t\geq 0
$$
and the process
\begin{equation}\label{sec4-eq4.1-001}
{\mathcal C}^{-,H}_t(a):=2\left( F_{-}(B^H_t-a)-F_{-}(-a)-\int_0^tF_{-}'(B^H_s-a)
\delta B^H_s\right),\quad t\geq 0
\end{equation}
exist in $L^2(\Omega)$ for all $a\in {\mathbb R}$, where
\begin{equation}\label{sec3-eq3.1000}
F_{-}(x)=
\begin{cases}
0,& {\text {if $x\geq 0$}},\\
x\log(-x)-x,& {\text {if $x<0$}}.
\end{cases}
\end{equation}
\begin{proposition}\label{lem3.3}
For all $a\in {\mathbb R}$ the integral
\begin{equation}\label{sec3-eq3.6}
X_t^{H}(a):=\int_0^t\log|B^H_s-a|\delta B^H_s,
\end{equation}
and
\begin{equation}\label{sec3-eq3.802000}
\begin{split}
E\left|\int_0^t\log|B^H_s-a|\delta B^H_s\right|^2&
=\int_0^t\int_0^tE[\log|B^H_s-a|\log|B^H_r-a|]\phi(s,r)dsdr\\
&\hspace{-2cm}+\int_0^tds\int_0^sd\xi\int_0^tdr\int_0^rd\eta
\phi(s,\eta)\phi(r,\xi)\int_{{\mathbb R}^2}
\frac{\Psi_{s,r,a,a}(x,y)dxdy}{(x-a)(y-a)}
\end{split}
\end{equation}
for all $t\geq 0$ and $a\in {\mathbb R}$ and the process
\begin{equation}\label{sec4-eq4.2}
{\mathcal C}^{H}_t(a):=2\left(F(B^H_t-a)-F(-a)-\int_0^t\log|B^H_s-a|\delta B^H_s\right),\quad t\geq 0
\end{equation}
are well defined, where $F(x)=x\log|x|-x$
\end{proposition}
\begin{proof}
Clearly, $F'(x)=\log|x|$, and the proposition follows from $F'=F'_{+}+F'_{-}$.
\end{proof}

\section{A representation of the functional ${\mathcal C}^H(a)$}
\label{sec4}

In this section we will consider the representation of the
functionals ${\mathcal C}^{+,H}(a),{\mathcal C}^{-,H}(a)$ and ${\mathcal C}^H(a)$, which point out that  $\frac1\pi{\mathcal C}_t^H(\cdot)$ is the Hilbert transform of weighted local time ${\mathscr L}^H(\cdot,t)$.

\begin{lemma}\label{lem4.1}
For any $0<\varepsilon<1$, $0<r<s\leq t$ and $\beta\in (0,1)$ we have
\begin{equation}\label{sec4-eq4.60}
\begin{split}
\Lambda_3(s,r,a):&=\int_a^{a+\varepsilon}\int_a^{a+\varepsilon}
\left(\log{(x-a)}-(\frac1\varepsilon\log\varepsilon)(x-a)\right)\\
&\qquad\qquad \cdot
\left(\log{(y-a)}-(\frac1\varepsilon\log\varepsilon)(y-a)\right)
\varphi_{s,r}(x,y)dxdy\\
&\leq C_H(sr)^{-H/2}\varepsilon^H
\end{split}
\end{equation}
and
\begin{equation}\label{sec4-eq4.70}
\begin{split}
\Lambda_4(s,r,a):&=\int_a^{a+\varepsilon}\int_a^{a+\varepsilon}
\left(\frac1{x-a}-\frac1\varepsilon\log\varepsilon\right)
\left(\frac1{y-a}-\frac1\varepsilon\log\varepsilon\right)
|\Psi_{s,r,a,a}(x,y)|dxdy\\
&\leq C_{H,t,\beta}\frac{s^{\beta H/2}}{r^{(1+\frac{\beta}2)H}(s-r)^{(1+\beta)H}} \varepsilon^{\beta}(1+\log^2\varepsilon).
\end{split}
\end{equation}
\end{lemma}
\begin{proof}
The estimate~\eqref{sec4-eq4.60} is clear. In order to prove~\eqref{sec4-eq4.70}, we have
\begin{align*}
\int_a^{a+\varepsilon}\int_a^{a+\varepsilon} &\left(\frac1{x-a}-\frac1\varepsilon\log\varepsilon\right)
\left(\frac1{y-a}-\frac1\varepsilon\log\varepsilon\right)\\
&\qquad\qquad\cdot[(x-a)(y-a)]^{\beta/2}dxdy \leq C_\beta\varepsilon^{\beta}(1+\log^2\varepsilon)
\end{align*}
for all $\beta\in (0,1)$, which gives
\begin{align*}
\int_a^{a+\varepsilon}\int_a^{a+\varepsilon}&
\left(\frac1{x-a}-\frac1\varepsilon\log\varepsilon\right)
\left(\frac1{y-a}-\frac1\varepsilon\log\varepsilon\right)
|\Psi_{s,r,a,a}(x,y)|dxdy\\
&=\int_a^{a+\varepsilon}\int_a^{a+\varepsilon}
\left(\frac1{x-a}-\frac1\varepsilon\log\varepsilon\right)
\left(\frac1{y-a}-\frac1\varepsilon\log\varepsilon\right)\\
&\qquad\qquad
\cdot |\varphi_{s,r}(x,y)
-\varphi_{s,r}(x,b)-\varphi_{s,r}(a,y) +\varphi_{s,r}(a,b)|dxdy\\
&\leq C_{H,t,\beta}\frac{s^{\beta H/2}}{r^{(1+\frac{\beta}2)H}(s-r)^{(1+\beta)H}} \varepsilon^{\beta}(1+\log^2\varepsilon)
\end{align*}
by~\eqref{lem3.1-eq1} and Lemma~\ref{lem2.10}. This completes the proof.
\end{proof}

\begin{lemma}\label{lem4.2}
Let $\frac12<H<1$ and $M>0$. We then have
\begin{equation}\label{sec4-eq4.80}
E\left|{\mathscr L}^H(b,t)-{\mathscr L}^H(a,t)\right|^2\leq C_{H,\alpha,t,M}|b-a|^{\alpha}
\end{equation}
for all $0<\alpha<\frac{1-H}{H}$, $t\geq 0$ and $a,b\in [-M,M]$.
\end{lemma}

The lemma is a direct consequence of H\"older continuity of $x\mapsto {\mathscr L}^H(x,t)$. Here, we shall use other method to prove it.
\begin{proof}[Proof of Lemma~\ref{lem4.2}]
Without loss of generality we may assume that $0<a<b$. Define the function $f_{a,b}(x)=1_{(a,b]}(x)$ and denote
$$
\widetilde{B}_t^H(x):=\int_0^t1_{\{B_s^H>x\}}\delta B_s^H
$$
and
$$
\psi_t(x):=(B_t^H-x)^{+}-(-x)^{+}
$$
for all $x\in {\mathbb R}$. Then the function $x\mapsto \psi_t(x)$ is Lipschitz continuous with Lipschitz constant $2$, and we have
$$
\left|\psi_t(x)-\psi_t(y)\right|\leq 2|x-y|
$$
for all $x,y\in {\mathbb R}$ and
$$
{\mathscr L}^{H}(x,t)=2\left(\psi_t(x) -\widetilde{B}_t^H(x)\right).
$$
by Tanaka formula, which deduces
\begin{align*}
E\left|{\mathscr L}^{H}(b,t)-{\mathscr L}^{H}(a,t)\right|^2&\leq 8(b-a)^2+4E\left|\widetilde{B}_t^H(b)-\widetilde{B}_t^H(a)\right|^2
\end{align*}
for all $t\geq 0$.

On the other hand, similar to the proof of~\eqref{sec3-eq3.802} by approximating the function $f(x)=1_{(a,b]}(x)$ by smooth functions we can obtain
\begin{equation}\label{sec4-eq4.90}
\begin{split}
E\left(\widetilde{B}_t^H(b)-\widetilde{B}_t^H(a)\right)^2
&=E\left(\int_0^tf_{a,b}(B^H_s)\delta B^H_s\right)^2\\
&=\int_0^t\int_0^tE\left[f_{a,b}(B^H_s)f_{a,b}(B^H_r)\right] \phi(s,r)dsdr\\
&\qquad +\int_0^tds\int_0^sd\xi\int_0^tdr\int_0^rd\eta
\Lambda_5(s,r,a,b)\phi(s,\eta)\phi(r,\xi)\\
&\equiv G_1(a,b)+G_2(a,b)
\end{split}
\end{equation}
for all $\frac12<H<1$, where
\begin{align*}
\Lambda_5(s,r,a,b)&=\varphi_{s,r}(a,a)-\varphi_{s,r}(a,b)-\varphi_{s,r}(b,a)
+\varphi_{s,r}(b,b).
\end{align*}
For the first term, we have
\begin{align*}
E\left[f_{a,b}(B^H_s)f_{a,b}(B^H_r)\right]&=\int_a^b\int_a^b
\frac{1}{2\pi\rho_{s,r}}\exp\left(-\frac1{2\rho^2_{s,r}} (r^{2H}x^2-2\mu_{s,r}xy+s^{2H}y^2)\right)dxdy\\
&=\frac1{2\pi}
\int_{\frac{a}{r^{H}}}^{\frac{b}{r^{H}}}e^{-\frac12x^2}dx
\int_{\frac{ar^{H}-\mu_{s,r}x}{\rho_{s,r}}}^{\frac{br^{H}-\mu_{s,r}
x}{\rho_{s,r}}}e^{-\frac12y^2}dy\\
&\leq\frac1{\sqrt{2\pi}}\int_{\frac{a}{r^{H}}}^{\frac{b}{r^{H}}}
e^{-\frac12x^2}dx \left(\frac1{\sqrt{2\pi}}\int_{\frac{ar^{H}-\mu_{s,r}
x}{\rho_{s,r}}}^{\frac{br^{H}-\mu_{s,r}
x}{\rho_{s,r}}}e^{-\frac12y^2}dy\right)^\beta\\
&\leq \left(\frac{r^H(b-a)}{\rho_{s,r}}\right)^\beta
\int_{\frac{a}{r^{H}}}^{\frac{b}{r^{H}}}e^{-\frac12x^2}dx \leq\frac{r^{(\alpha-1)H}}{\rho_{s,r}^{\beta}}(b-a)^{1+\beta},
\end{align*}
for all $s,r>0$ and $\beta\in (0,1)$. It follows from  Lemma~\ref{lem2.10} that
\begin{equation}\label{sec4-eq4.90-1}
\begin{split}
G_1(a,b)&=\int_0^t\int_0^tE\left[f_{a,b}(B^H_s)f_{a,b}(B^H_r)\right] \phi(s,r)dsdr\\
&\leq C_{H,\beta}t^{H(1-\beta)}(b-a)^{1+\beta}
\end{split}
\end{equation}
for all $\frac12<H<1$ and $0\leq \beta<\frac{2H-1}H$.

For the second term, we have also by~\eqref{sec3-eq3.8111} and Lemma~\ref{lem2.10}
\begin{align*}
G_2(a,b)&=\int_0^tds\int_0^sd\xi\int_0^tdr\int_0^rd\eta
\Lambda_5(a,b,s,r)\phi(s,\eta)\phi(r,\xi)\\
&\leq C_H(b-a)^{\alpha}\int_0^t\int_0^t\frac{(sr)^{2H-1}r^{\alpha H}}{ \rho_{s,r}^{1+\alpha}}drds \leq C_{H,\alpha}(b-a)^{\alpha}t^{H(2-\alpha)}
\end{align*}
for all $0<\alpha<\frac{1-H}{H}$, and the lemma follows.
\end{proof}

The main object of this section is to prove the following theorem.
\begin{theorem}\label{th4.1}
The convergence
\begin{equation}\label{sec4-eq4.5}
{\mathcal C}^{+,H}_t(a)=\lim_{\varepsilon\downarrow
0}\left\{(\log\varepsilon){\mathscr L}^H(a,t)+\int_0^t1_{\{B^H_s-a\geq
\varepsilon\}} \frac{2Hs^{2H-1}}{B^H_s-a}ds\right\}
\end{equation}
holds in $L^2(\Omega)$ for all $t\geq 0$.
\end{theorem}

\begin{proof}
Let $t\geq 0$ and $a\in {\mathbb R}$. We split the proof in three steps.

{\bf Step I.} Define the function $F_\varepsilon$ as follows
$$
F_\varepsilon(x)=
\begin{cases}
0,& {\text {$x\leq 0$}},\\
\frac1{2\varepsilon}(x^2\log\varepsilon),& {\text {$0<x\leq \varepsilon$}},\\
\varepsilon-\frac1{2}(\varepsilon\log\varepsilon)+x\log{x}-x,&
{\text {$x>\varepsilon$}}.
\end{cases}
$$
Then $F_\varepsilon\in C^1({\mathbb R})$, and
$$
F'_{\varepsilon}(x)=
\begin{cases}
0,& {\text {$x\leq 0$}},\\
\frac1{\varepsilon}(x\log\varepsilon),& {\text {$0<x\leq \varepsilon$}},\\
\log x,& {\text {$x>\varepsilon$}},
\end{cases}
\qquad
F''_{\varepsilon}(x)=
\begin{cases}
0,& {\text {$x<0$}},\\
\frac1{\varepsilon} \log\varepsilon,& {\text {$0<x<\varepsilon$}},\\
\frac1x,& {\text {$x>\varepsilon$}}
\end{cases}
$$
for all $\varepsilon\in (0,1)$. We shall show that the It\^o formula
\begin{equation}\label{sec4-eq4.7}
\begin{split}
F_\varepsilon(B^H_t-a)&-F_\varepsilon(-a)-\int_0^t F'_\varepsilon(B^H_s-a)\delta B^H_s\\
&=\frac{\log\varepsilon}{\varepsilon}\int_0^t1_{\{0\leq B^H_s-a
<\varepsilon\}}Hs^{2H-1}ds+\int_0^t1_{\{B^H_s-a\geq
\varepsilon\}}\frac{Hs^{2H-1}}{B^H_s-a}ds
\end{split}
\end{equation}
holds for $\varepsilon\in (0,1)$. Define the sequence of smooth functions
\begin{align}\label{Ito-formula-2}
f_{n,\varepsilon}(x):=\int_{\mathbb
R}F_{\varepsilon}(x-{y})\zeta_n(y)dy=\int_0^2 F_{\varepsilon}(x-\frac{y}n)\zeta(y)dy,\qquad
n=1,2,\ldots
\end{align}
for all $\varepsilon\in (0,1)$, where $\zeta$ is defined by~\eqref{Ito-formula-1} and $\zeta_n(x)=n\zeta(nx)$. Then $f_{n,\varepsilon}\in C^\infty_0({\mathbb R})$ and
\begin{align*}
f_{n,\varepsilon}(B^H_t-a)=f_{n,\varepsilon}(-a) +\int_0^tf'_{n,\varepsilon}(B^H_s-a)\delta B^H_s +H\int_0^tf''_{n,\varepsilon}(B^H_s-a)s^{2H-1}ds
\end{align*}
for all $n=1,2,\ldots$. Notice that
$$
f_{n,\varepsilon}(x)\longrightarrow F_{\varepsilon}(x),\quad f'_{n,\varepsilon}(x)\longrightarrow F'_{\varepsilon}(x)
$$
as $n\to \infty$, uniformly in $\mathbb R$, and
$$
|f''_{n,\varepsilon}(x)|\leq \frac{1}{\varepsilon}|\log\varepsilon|,\qquad \forall x\in {\mathbb R},
$$
and $f''_{n,\varepsilon}(x)\to F''_{\varepsilon}(x)$ pointwise (besides $0$ and $\varepsilon$), as $n\to \infty$ by Lebesgue's dominated convergence theorem. We get
\begin{align*}
\int_0^tf'_{n,\varepsilon}&(B^H_s-a)\delta B^H_s= f_{n,\varepsilon}(B^H_t-a)-f_{n,\varepsilon}(-a)- H\int_0^tf''_{n,\varepsilon}(B^H_s-a)s^{2H-1}ds\\
&\longrightarrow F_{\varepsilon}(B^H_t-a)-F_{\varepsilon}(-a)- H\int_0^tF''_{\varepsilon}(B^H_s-a)s^{2H-1}ds\qquad {\text { in $L^2(\Omega)$}}\\
&=F_{\varepsilon}(B^H_t-a)-F_{\varepsilon}(-a)
-H\varepsilon^{-1}\int_0^t1_{\{0<B^H_s-a<\varepsilon\}} s^{2H-1}ds\qquad {\text { a.s.}},
\end{align*}
as $n\to \infty$, which implies that It\^o's formula
\begin{equation}\label{sec4-eq4.6}
F_\varepsilon(B^H_t-a)=F_\varepsilon(-a)+\int_0^t F'_\varepsilon(B^H_s-a)\delta B^H_s
+H\int_0^tF''_\varepsilon(B^H_s-a)s^{2H-1}ds
\end{equation}
holds for all $\varepsilon\in (0,1)$. This gives~\eqref{sec4-eq4.7}.

{\bf Step II.} We show that the limit
\begin{equation}\label{sec4-eq4.8}
\lim_{\varepsilon\downarrow
0}\left\{F_\varepsilon(B^H_t-a)-F_\varepsilon(-a) -\int_0^tF'_\varepsilon(B^H_s-a)\delta B^H_s
\right\}
\end{equation}
exists in $L^2(\Omega)$, and is equal to
$$
\frac12{\mathcal C}^{+,H}_t(a)=F_{+}(B^H_t-a)-F_{+}(-a)-\int_0^tF_{+}'(B^H_s-a)\delta B^H_s,
$$
where $F_{+}$ is given by~\eqref{sec3-eq3.4}. We have
\begin{equation}\label{sec4-eq4.11}
\begin{split}
E&\left|\frac12{\mathcal C}^{+,H}_t(a)+F_\varepsilon(-a)-F_\varepsilon(B^H_t-a)
+\int_0^tF'_\varepsilon(B^H_s-a)\delta B^H_s \right|^2\\
&\qquad\qquad\leq
3E\left|F(B^H_t-a)-F_\varepsilon(B^H_t-a)\right|^2 +3|F_\varepsilon(-a)-F_{+}(-a)|^2\\
&\hspace{3cm}
+3E\left|\int_0^t[F'(B^H_s-a)-F'_\varepsilon(B^H_s-a)]\delta B^H_s \right|^2.
\end{split}
\end{equation}
The first and second term of the right-hand side in~\eqref{sec4-eq4.11} tends to $0$ as $\varepsilon\to 0$ because
$$
|F_{+}(x)-F_\varepsilon(x)|\leq
\varepsilon-\frac12\varepsilon\log\varepsilon
$$
for all $\varepsilon\in (0,1)$. To estimate the third term, we consider the approximation of the function $F'_\varepsilon$ as follows
\begin{align*}
\widehat{G}_{n,\varepsilon}(x)&=\int_{\mathbb R}F'_\varepsilon(y)
\zeta_n(x-y)dy,\quad n\geq 2
\end{align*}
for all $\varepsilon\in (0,1)$, where $\zeta_n, n\geq 2$ is given by~\eqref{Ito-formula-101}. Then $G_{n,\varepsilon},n\geq 2$ are smooth functions with compact supports. Denote
$$
G_{n,\varepsilon}(x):=G_{n}(x)-\widehat{G}_{n,\varepsilon}(x)
$$
for $x\in {\mathbb R}$, where $G_n$ is defined by~\eqref{eq3-2020202}. Similar to proofs of Lemma~\eqref{lem3.301} and Lemma~\eqref{lem3.301-1}, we can obtain the next statements:
\begin{align*}
|G_{n,\varepsilon}(x)|\leq C\psi_1(x),\quad |G'_{n,\varepsilon}(x)|\leq C\psi_2(x)
\end{align*}
for all $x\in {\mathbb R},\varepsilon\in (0,1)$ and
$$
G_{n,\varepsilon}(x)\longrightarrow F_{+}'(x)-F'_\varepsilon(x),\quad G'_{n,\varepsilon}(x)\longrightarrow F_{+}''(x)-F''_\varepsilon(x)\quad (n\to \infty)
$$
for all $x\neq 0$ and $\varepsilon\in (0,1)$. Thus, in a same way as the proof of Proposition~\ref{lem3.2}, we can obtain
\begin{align*}
E&\left|\int_0^t[F_{+}'(B^H_s-a)-F'_\varepsilon(B^H_s-a)]\delta B^H_s \right|^2\\
&=\int_0^t\int_0^t\Lambda_3(s,r,a)\phi(s,r)dsdr
+\int_0^tds\int_0^sd\xi\int_0^tdr\int_0^rd\eta
\Lambda_4(s,r,a)\phi(s,\eta)\phi(r,\xi)\\
&=\int_0^t\int_0^t\Lambda_3(s,r,a)\phi(s,r)dsdr
+\int_0^t\int_0^t(sr)^{2H-1}drds \Lambda_4(s,r,a)\phi(s,\eta)\phi(r,\xi)\\
&\leq C_{H}(t^H+t^{2-(2+\beta)H})\varepsilon^{\beta\wedge H }(1+\log^2\varepsilon)\longrightarrow 0
\end{align*}
with $0<\beta<\frac{1-H}H$ by Lemma~\ref{lem4.1}. It follows from the It\^o formula~\eqref{sec4-eq4.7} that
\begin{equation}\label{sec4-eq4.12}
\begin{split}
{\mathcal C}^{+,H}_t(a)
&=2\lim_{\varepsilon\downarrow 0}\left\{F_\varepsilon(B^H_t-a)-F_\varepsilon(-a) -\int_0^tF'_\varepsilon(B^H_s-a)\delta B^H_s\right\}\\
&=\lim_{\varepsilon\downarrow 0}J_t^H(\varepsilon,a)
\end{split}
\end{equation}
in $L^2(\Omega)$, where
$$
J_t^H(\varepsilon,a)=\frac{\log\varepsilon}{\varepsilon}
\int_0^t1_{\{0\leq B^H_s-a
<\varepsilon\}}2Hs^{2H-1}ds+\int_0^t1_{\{B^H_s-a\geq
\varepsilon\}}\frac{2Hs^{2H-1}}{B^H_s-a}ds.
$$

{\bf Step III.} To end the proof, we decompose $J_t^H(\varepsilon,a)$ as
\begin{align*}
J_t^H(\varepsilon,a)=I_t(\varepsilon,a)+\left\{\int_0^t1_{\{B^H_s-a\geq
\varepsilon\}}\frac{2Hs^{2H-1}}{B^H_s-a}ds+(\log\varepsilon){\mathscr
L}^H(a,t)\right\},
\end{align*}
where
\begin{align*}
I_t(\varepsilon,a):=\frac{\log\varepsilon}{\varepsilon}
\int_0^t1_{\{0\leq B^H_s-a
<\varepsilon\}}2Hs^{2H-1}ds-(\log\varepsilon){\mathscr L}^H(a,t).
\end{align*}
According to Lemma~\ref{lem4.2} we get
\begin{align*}
E|I_t(\varepsilon,a)|^2:&=(\log\varepsilon)^2{E} \left|\frac{1}{\varepsilon}
\int_0^t1_{\{0\leq B^H_s-a<\varepsilon\}}2Hs^{2H-1}ds-{\mathscr
L}^H(a,t)\right|^2\\
&=(\log\varepsilon)^2{E}\left|\frac{1}{\varepsilon}
\int_0^{\varepsilon}{\mathscr L}^H(x+a,t)dx-{\mathscr
L}^H(a,t)\right|^2\\
&\leq(\log\varepsilon)^2\frac{1}{\varepsilon}
\int_0^{\varepsilon}{E}|{\mathscr L}^H(x+a,t)-{\mathscr
L}^H(a,t)|^2dx\\
&\leq C_{H,t,\alpha}\varepsilon^{\alpha} (\log\varepsilon)^2\longrightarrow 0\qquad\quad (\varepsilon\to 0)
\end{align*}
for all $0<\alpha<\frac{1-H}{H}$ and $t\geq 0$, which shows that
\begin{align*}
{\mathcal C}^{+,H}_t(a)=\lim_{\varepsilon\downarrow 0}\left\{\int_0^t1_{\{B^H_s-a\geq
\varepsilon\}}\frac{2Hs^{2H-1}}{B^H_s-a}ds+(\log\varepsilon){\mathscr
L}^H(a,t)\right\}\qquad {\text { in $L^2(\Omega)$}}
\end{align*}
for all $t\geq 0$, and the theorem follows.
\end{proof}

\begin{theorem}\label{th4.2}
The convergence
\begin{equation}\label{sec4-eq4.16}
{\mathcal C}^H_t(a)=\lim_{\varepsilon\downarrow 0}\int_0^t1_{\{|B^H_s-a|\geq \varepsilon\}}
\frac{2Hs^{2H-1}}{B^H_s-a}ds\equiv {\rm v.p.}\int_0^t\frac{2Hs^{2H-1}}{B^H_s-a}ds
\end{equation}
holds in $L^2(\Omega)$.
\end{theorem}
\begin{proof}
In the same way as the proof of~\eqref{sec4-eq4.5}, we can show that the convergence
\begin{equation}\label{sec4-eq4.17}
{\mathcal C}^{-,H}_t(a)
=\lim_{\varepsilon\downarrow
0}\left\{-(\log\varepsilon){\mathscr
L}^H(a,t)+\int_0^t1_{\{B^H_s-a\leq -\varepsilon\}}
\frac{2Hs^{2H-1}}{B^H_s-a}ds\right\}
\end{equation}
holds in $L^2(\Omega)$. Thus,~\eqref{sec4-eq4.16} follows from
$F=F_{+}+F_{-}$, where $F(x)=x\log|x|-x$.
\end{proof}
According to the occupation formula we get
\begin{equation}\label{sec4-eq4.18}
\begin{split}
{\mathcal C}^H_t(a)
&=\lim_{\varepsilon\downarrow 0}\int_0^t1_{\{|B^H_s-a|\geq
\varepsilon\}}
\frac{2Hs^{2H-1}}{B^H_s-a}ds\qquad {\text { in $L^2(\Omega)$}}\\
&=\lim_{\varepsilon\downarrow 0}\int_{\mathbb
R}1_{\{|x-a|\geq \varepsilon\}}\frac{{\mathscr L}^H(x,t)}{x-a}dx\qquad {\text { in $L^2(\Omega)$}}\\
&={\rm v.p.}\int_{\mathbb
R}\frac{{\mathscr L}^H(x,t)}{x-a}dx=\pi\left({\mathscr H}{\mathscr L}^H(\cdot,t)\right)(a)
\end{split}
\end{equation}
for all $t\geq 0$ and $a\in {\mathbb R}$. As two natural results we get the fractional version of Yamada's formula
\begin{align*}
(B^H_t-a)&\log|B^H_t-a|-(B^H_t-a)\\
&=-a\log|a|+a+\int_0^t\log|B^H_s-a|\delta B^H_s+\frac12{\rm v.p.}\int_{\mathbb R}\frac{{\mathscr L}^H(x,t)}{x-a}dx
\end{align*}
for all $t\geq 0$ and $a\in {\mathbb R}$, and
\begin{align*}
{\mathcal C}^H_t(b)&-{\mathcal C}^H_s(a)\\
&=\int_0^\infty\left[{\mathscr L}^H(b+x,t)-{\mathscr L}^H(b-x,t)-{\mathscr L}^H(a+x,s)+{\mathscr L}^H(a-x,s)\right]\frac{dx}{x}
\end{align*}
for all $a,b\in {\mathbb R}$ and $s,t\geq 0$. Recall that the local time ${\mathcal L}^H(x,t)$ admits a compact support and it is H\"older continuous of order $\gamma\in (0,1-H)$ in time, and of order $\kappa\in (0,\frac{1-H}{2H})$ in the space variable (see Geman-Horowitz~\cite{Geman}). We see that the process $(a,t)\mapsto {\mathcal C}^H_t(a)$ admits H\"older continuous paths. In particular, we have
\begin{proposition}\label{prop9.1}
Let $\frac12<H<1$. For all $t'>t\geq 0$, we have
$$
E\left[|{\mathcal C}^{H}_{t'}-{\mathcal C}^{H}_t|^2\right] \leq C(t'-t)^{2H_0},
$$
where
$$
H_0=
\begin{cases}
H, & {\text { if $\frac12<H\leq \frac23$}},\\
1-\frac12H, & {\text { if $\frac23<H<1$}}.
\end{cases}
$$
\end{proposition}
\begin{proof}
Given $\varepsilon>0$ and denote
$$
{\mathcal C}^{H,\varepsilon}_t=\int_0^t1_{\{|B^H_s|>\varepsilon\}}\frac{ds^{2H}}{B^H_s}
$$
for $t\geq 0$. We have
\begin{align*}
E&\left[1_{\{|B^H_s|>\varepsilon\}}1_{\{|B^H_r|>\varepsilon\}} \frac{1}{B^H_sB^H_r}\right]=\int_{{\mathbb R}^2}
1_{\{|x|>\varepsilon\}}1_{\{|y|>\varepsilon\}} \frac{1}{xy}\varphi_{s,r}(x,y)dxdy\\
&=\int_\varepsilon^\infty\int_\varepsilon^\infty
\frac{1}{xy}\varphi_{s,r}(x,y)dxdy+\int_{-\infty}^{-\varepsilon} \int_\varepsilon^\infty\frac{1}{xy}\varphi_{s,r}(x,y)dxdy\\
&\qquad+\int_\varepsilon^\infty\int_{-\infty}^{-\varepsilon}
\frac{1}{xy}\varphi_{s,r}(x,y)dxdy+\int_{-\infty}^{-\varepsilon}
\int_{-\infty}^{-\varepsilon}\frac{1}{xy}\varphi_{s,r}(x,y)dxdy\\
&=\int_\varepsilon^\infty\int_\varepsilon^\infty
\frac{1}{xy}\left[\varphi_{s,r}(x,y)-\varphi_{s,r}(-x,y)-\varphi_{s,r}(x,-y) +\varphi_{s,r}(-x,-y)\right]dxdy\\
&=2\int_\varepsilon^\infty\int_\varepsilon^\infty
\frac{1}{xy}\left[\varphi_{s,r}(x,y)-\varphi_{s,r}(-x,y)\right]dxdy\\
&=2\int_\varepsilon^\infty\int_\varepsilon^\infty
\frac{1}{xy}\left(1-e^{
-\frac{1}{\rho^2_{s,r}}\mu_{s,r}xy}\right)\varphi_{s,r}(x,y)dxdy\\
&=2\int_\varepsilon^\infty\int_\varepsilon^\infty \left(\int_0^{\frac{\mu_{s,r}}{\rho^2_{s,r}}} e^{-xy\xi }d\xi\right)\varphi_{s,r}(x,y)dxdy\\
&=2\int_0^{\frac{\mu_{s,r}}{\rho^2_{s,r}}}d\xi \int_\varepsilon^\infty\int_\varepsilon^\infty e^{-xy\xi }\varphi_{s,r}(x,y)dxdy
\end{align*}
for all $s,t\geq 0$. An elementary calculus can show that
\begin{align*}
\int_0^\infty&\int_0^\infty e^{-xy\xi }\varphi_{s,r}(x,y)dxdy\\
&=\frac1{2\pi\rho_{s,r}} \int_0^\infty e^{-\frac1{r^{2H}}(1+2\mu_{s,r}\xi-\rho_{s,r}^2\xi^2)y^2}dy\int_0^\infty
e^{-\frac{r^{2H}}{2\rho_{s,r}^2}\left(x -\frac1{r^{2H}}(\mu_{s,r}-\rho^2_{s,r}\xi)y\right)^2}dx\\
&=\frac1{4\sqrt{1+2\mu_{s,r}\xi-\rho_{s,r}^2\xi^2}}
\end{align*}
for all $\xi>0$, which implies that
\begin{align*}
E&\left[1_{\{|B^H_s|>\varepsilon\}}1_{\{|B^H_r|>\varepsilon\}} \frac{1}{B^H_sB^H_r}\right]=\int_{{\mathbb R}^2}
1_{\{|x|>\varepsilon\}}1_{\{|y|>\varepsilon\}} \frac{1}{xy}\varphi_{s,r}(x,y)dxdy\\
&\qquad\qquad\leq \int_0^{\frac{\mu_{s,r}}{\rho_{s,r}^2}} \frac1{4\sqrt{1+2\mu_{s,r}\xi-\rho_{s,r}^2\xi^2}}d\xi
=\frac1{4\rho_{s,r}}\arcsin\frac{\mu_{s,r}}{\sqrt{\rho_{s,r}^2+\mu_{s,r}^2}}\\
&\qquad\qquad=\frac1\rho_{s,r}\arcsin\frac{\mu_{s,r}}{(sr)^H}\leq \frac1\rho_{s,r}.
\end{align*}
It follows that
\begin{align*}
E\left[|{\mathcal C}^{H,\varepsilon}_{t'}-{\mathcal C}^{H,\varepsilon}_t|^2\right]&\leq \int_t^{t'}
\int_t^{t'}\frac1{\rho_{s,r}}ds^{2H}dr^{2H} \leq
\begin{cases}
C(t'-t)^{2H}, & {\text { if $\frac12<H\leq \frac23$}},\\
C(t-t')^{2-H}, & {\text { if $\frac23<H<1$}}
\end{cases}
\end{align*}
for all $0<t<t'<T$ and $\varepsilon>0$. This shows that
$$
E\left[|{\mathcal C}^{H}_{t'}-{\mathcal C}^{H}_t|^2\right] \leq C(t'-t)^{2H_0}
$$
and the proposition follows.
\end{proof}
\begin{remark}
{\rm

The above continuity results for the process $(x,t)\mapsto {\mathcal C}^H(x,t):={\mathcal C}^H_t(x)$ are some reminders to us that we may consider the following integrals:
$$
\int_0^tu_sd{\mathcal C}^H_s,\quad \int_{\mathbb R}f(x){\mathcal C}^H(dx,t),\quad \int_0^t\int_{\mathbb R}f(x,s){\mathcal C}^H(dx,ds),
$$
where $u$ is an adapted process, and $(x,t)\mapsto f(x,t)$ and $x\mapsto f(x)$ Borel functions on ${\mathbb R}\times [0,T]$ and ${\mathbb R}$, respectively. These will be considered in the other paper.
}
\end{remark}

\section{The occupation formula associated with ${\mathcal C}^H(a)$}
\label{sec5}
From the previous sections we know that the process $(a,t)\mapsto {\mathcal C}^H_t(a)$ is H\"older continuous and in this section our main object is to expound and prove the next theorem which is an analogue of the occupation formula.
\begin{theorem}\label{th5.3}
Let $\frac12<H<1$ and let $g$ be a continuous function with compact support. We then have, almost surely,
\begin{equation}\label{sec5-eq5.26}
\int_{\mathbb R}{\mathcal C}^H_t(x)g(x)dx=2H\pi\int_0^t({\mathscr
H}g)(B^H_s)s^{2H-1}ds
\end{equation}
and
$$
2H\pi\int_0^tf(B^H_s)s^{2H-1}ds=\int_{\mathbb R}{\mathcal
C}^H_t(x)({\mathscr H}^{-1}f)(x)dx
$$
for all $t\geq 0$, where the operator ${\mathscr H}^{-1}$ means the
inverse transform of Hilbert transform ${\mathscr H}$.
\end{theorem}

In order to prove the theorem we need some preliminaries.
\begin{lemma}\label{lem5.3}
Let $F(x)=x\log|x|-x$ and let $g$ be a continuous function with compact support. Then the integral
$$
\int_0^t(F'\ast g)(B^H_s)\delta B^H_s
$$
exists in $L^2(\Omega)$ for all $t\geq 0$ and the process
$$
{\mathcal X}^{g}_t:=(F\ast g)(B^H_t)-(F\ast g)(0)-\int_0^t(F'\ast g)(B^H_s)\delta B^H_s,\quad t\geq 0
$$
is well-defined.
\end{lemma}
\begin{proof}
From Lemma~\ref{lem3.-1} it follows that
\begin{align*}
E&\left|\int_0^t(G\ast g)(B^H_s)\delta B^H_s\right|^2
=\int_0^t\int_0^tE\left[(G\ast g)(B^H_s)(G\ast g)(B^H_r)\right]\phi(s,r)dsdr\\
&\qquad+\int_0^tds\int_0^tdr\int_0^sd\xi\int_0^rd\eta \phi(s,\eta)\phi(r,\xi)E\left[(G'\ast g)(B^H_s)(G'\ast g)(B^H_r)\right]\\
&=\int_0^t\int_0^t\phi(s,r)dsdr\int_{{\mathbb R}^2}g(u)g(v)dudv \int_{{\mathbb R}^2}G(x-u)G(y-v)\varphi_{s,r}(x,y)dxdy\\
&\qquad+\int_0^tds\int_0^tdr\int_0^sd\xi\int_0^rd\eta \phi(s,\eta)\phi(r,\xi)\\
&\hspace{2cm}\cdot\int_{{\mathbb R}^2}g(u)g(v)dudv\int_{{\mathbb R}^2}G'(x-u)G'(y-v)\Psi_{s,r,u,v}(x,y)dxdy\\
&\qquad+\int_0^tds\int_0^tdr\int_0^sd\xi\int_0^rd\eta \phi(s,\eta)\phi(r,\xi)
\int_{{\mathbb R}^2}g(u)g(v)\Lambda_7(s,r,u,v)dudv
\end{align*}
for all $t>0,u,v\in {\mathbb R}$, and all $G\in C^\infty({\mathbb R})$ with compact support, where
\begin{align*}
\Lambda_7(s,&r,u,v)=-G(1)\int_{{\mathbb R}}G(x-u)\frac{\partial}{\partial x}\varphi_{s,r}(x,v)dx\\
&\hspace{2cm}-G(1)\int_{{\mathbb R}}G(y-v)\frac{\partial}{\partial y}\varphi_{s,r}(u,y)dy-\varphi_{s,r}(u,v)G(1)G(1).
\end{align*}
Decompose $F$ as
$$
F(x)=F_{+}(x)+F_{-}(x),
$$
where $F_{+}$ and $F_{-}$ are given in Section~\ref{sec3}. Clearly, we have
\begin{align*}
\int_{{\mathbb R}^2}&|g(u)g(v)|dudv\int_{{\mathbb R}^2}|F'_{+}(x-u)F'_{+}(y-v)|\varphi_{s,r}(x,y)dxdy\\
&\leq \int_{{\mathbb R}^2}|g(u)g(v)|dudv\left(\int_u^\infty\log^2(x-u)\varphi_s(x)dx
\int_v^\infty\log^2(y-v)\varphi_r(y)dy\right)^{1/2}\\
&\leq \int_{{\mathbb R}^2}|g(u)g(v)|dudv\left(s^{-H}+s^H+|u|\right)^{1/2}
\left(r^{-H}+r^H+|v|\right)^{1/2}\\
&\leq C_H\left(r^{-H}+s^H+1\right)\left(\int_{\mathbb R}|g(u)|(\sqrt{|u|}+1)du\right)^2,
\end{align*}
and
\begin{align*}
\int_{{\mathbb R}^2}|g(u)g(v)|dudv&\int_u^{\infty}dx\int_v^\infty
\frac{|\Psi_{s,r,u,v}(x,y)|dy}{(x-u)(y-v)}\\
&\leq C_{H,T,\beta}\frac{ s^{\beta H/2}}{r^{(1+\beta)H}(s-r)^{(1+\beta)H}} \int_{{\mathbb R}^2}|g(u)g(v)|dudv
\end{align*}
by Lemma~\ref{lem3.1}. Thus, similar to the proof of Proposition~\ref{lem3.2} by approximating the function $F_{+}'(x)$ by smooth functions with compact support, we can show that the integral $\int_0^t(F_{+}'\ast g)(B^H_s)\delta B^H_s$ exists in $L^2(\Omega)$ for all $t\geq 0$ and
\begin{align*}
E&\left|\int_0^t(F_{+}'\ast g)(B^H_s)\delta B^H_s\right|^2=\int_0^t\int_0^tdsdr\phi(s,r)\int_{{\mathbb R}^2}dudvg(u)g(v)\\
&\hspace{5cm}\cdot\int_{{\mathbb R}^2}F'_{+}(x-u)F'_{+}(y-v)\varphi_{s,r}(x,y)dxdy\\
&\hspace{4cm}\;+\int_0^tds\int_0^sd\xi\int_0^tdr\int_0^rd\eta
\phi(s,\eta)\phi(r,\xi)\\
&\hspace{5cm}\cdot\int_{{\mathbb R}^2}g(u)g(v)dudv\int_u^{\infty}dx\int_v^\infty
\frac{\Psi_{s,r,u,v}(x,y)dy}{(x-u)(y-v)}.
\end{align*}
Similarly, we can also show that the integral $\int_0^t(F_{-}'\ast g)(B^H_s)\delta B^H_s$ exists in $L^2(\Omega)$ for all $t\geq 0$, and the lemma follows since $F=F_{+}+F_{-}$.
\end{proof}

\begin{lemma}\label{lem5.1}
Let $F(x)=x\log|x|-x$ and let $g$ be a continuous function with compact support. Then
\begin{equation}\label{sec5-eq5.25}
\int_0^t\left(\int_{\mathbb R}F'(B^H_s-x)g(x)dx\right)\delta B^H_s
=\int_{\mathbb R}\left(\int_0^tF'(B^H_s-x)\delta B^H_s\right)g(x)dx
\end{equation}
for all $0\leq t\leq T$.
\end{lemma}
By using the divergence operator $\delta^{H}$ we can rewrite~\eqref{sec5-eq5.25} as
$$
\delta^H(F'\ast g(B^H))=\int_{\mathbb R}g(a)da\int_0^tF'(B^H-a)\delta B^H_s.
$$
\begin{proof}[Proof of Lemma~\ref{lem5.1}]
Clearly, we have
$$
F'\ast g=(F\ast g)',\qquad (F\ast g)''={\rm v.p.}\frac1x\ast g.
$$
Moreover, the functional
$$
x\mapsto \int_0^tF'(B^H_s-x)\delta B^H_s
$$
is Borel measurable for every $t\geq 0$ and the right-hand side in~\eqref{sec5-eq5.25} exists also in $L^2(\Omega)$ by Proposition~\ref{lem3.2}.

Denote by $X$ the process concerning the right hand in~\eqref{sec5-eq5.25} and let
$$
u_t=\int_{\mathbb R}F'(B^H_t-x)g(x)dx
$$
for $t\geq 0$. Then, the process $X$ and $u$ are measurable. Thus, it is enough to show that the following duality relationship holds:
\begin{equation}\label{sec5-eq5.25-1}
E\left[UX_T\right]=E\left[\langle D^HU,u\rangle_{\mathscr H}\right]
\end{equation}
for all $U\in {\mathbb D}^{1,2}$ by Lemma~\ref{lem5.3}. This is clear. In fact, noting that
\begin{align*}
\int_{\mathbb R}\Bigl(\int_0^T\int_0^T(D^H_sU)&F'(B^H_r-x)\phi(s,r)dsdr\Bigr) g(x)dx\\
&=\int_0^T\int_0^T(D^H_sU)\Bigl(\int_{\mathbb R}F'(B^H_r-x)g(x)dx\Bigr)\phi(s,r)dsdr,\qquad {\rm a.s.}
\end{align*}
for all $U\in {\mathbb D}^{1,2}$, we have
\begin{align*}
E[UX_T]&=E\left[U\int_{\mathbb R}\left(\int_0^TF'(B^H_s-x)\delta B^H_s\right)g(x)dx\right]\\
&=\int_{\mathbb R}E\left[U\int_0^TF'(B^H_s-x)\delta B^H_s\right]g(x)dx\\
&=\int_{\mathbb R}\left(E\langle D^HU,F'(B^H-x)\rangle_{\mathcal H}\right)g(x)dx\\
&=E\int_{\mathbb R}\Bigl(\int_0^T\int_0^T(D^H_sU)F'(B^H_r-x) \phi(s,r)dsdr\Bigr)g(x)dx\\
&=E\int_0^T\int_0^T(D^H_sU)\Bigl(\int_{\mathbb R}F'(B^H_r-x)g(x)dx\Bigr)\phi(s,r)dsdr=E\left[\langle D^HU,u\rangle_{\mathcal H}\right]
\end{align*}
for all $U\in {\mathbb D}^{1,2}$, and the lemma follows.
\end{proof}
\begin{proof}[Proof of Theorem~\ref{th5.3}]
Let $F(x)=x\log|x|-x$. Then second derivative $(F\ast g)''=F''\ast g$ exists in the sense of Schwartz's distribution, and similar to Theorem~\ref{th4.1} we have
\begin{align*}
{\mathcal X}^{g}_t&=H \int_0^t{\rm v.p.}\frac1{x}\ast g(B^H_s)s^{2H-1}ds
\end{align*}
for all $t\geq 0$, where ${\mathcal C}^{g}$ is defined in Lemma~\ref{lem5.3}. It follows from Lemma~\ref{lem5.1} that
\begin{align*}
\frac12\int_{\mathbb R}{\mathcal C}^H_t(x)g(x)dx&=\int_{\mathbb R}\left(F(B^H_s-x)-F(-x)-\int_0^tF'(B^H_s-x)\delta B^H_s\right)g(x)dx\\
&=F\ast g(B^H_t)-F\ast g(0)-\int_0^tF'\ast g(B^H_s)\delta B^H_s\\
&=F\ast g(B^H_t)-F\ast g(0)-\int_0^t(F\ast g)'(B^H_s)\delta B^H_s\\
&=H\int_0^t{\rm v.p.}\frac1x\ast g (B^H_s)s^{2H-1}ds\\
&=H\pi\int_0^t{\mathscr H}g(B^H_s)s^{2H-1}ds
\end{align*}
for all $t\geq 0$. This completes the proof.
\end{proof}
\begin{corollary}
Let $\frac12<H<1$ and let $g,g_n\in L^2({\mathbb R})$ be continuous with compact supports. If $g_n\to g$ in $L^2({\mathbb R})$, as $n$ tends to infinity, we then have
\begin{equation}\label{sec5-eq5.27}
\lim_{n\to \infty}\int_{\mathbb R}{\mathcal
C}^H_t(x)g_n(x)dx=\int_{\mathbb R}{\mathcal C}^H_t(x)g(x)dx
\end{equation}
for all $t\geq 0$, in the $L^2(\Omega)$.
\end{corollary}
\begin{proof}
The convergence follows from the identity
$$
\int_{\mathbb R}\left(g_n(x)-g(x)\right)^2dx
=\int_{\mathbb R}\left({\mathscr H}g_n(x)-{\mathscr H}g(x)\right)^2dx.
$$
and Theorem~\ref{th5.3}.
\end{proof}



\section{The case $0<H<\frac12$}\label{sec7}

In the final section we consider the process ${\mathcal C}^H$ with $0<H<\frac12$. Recall that for $0<H<\frac12$, Yan {\em et al}~\cite{Yan7} obtained the {\it generalized quadratic covariation} of $f(B^H)$ and $B^H$ defined by
$$
[f(B^H),B^H]^{(H)}_t:=\lim_{\varepsilon\downarrow
0}\frac{1}{\varepsilon^{2H}}\int_0^t\left\{f(B^{H}_{
s+\varepsilon})-f(B^{H}_s)\right\}(B^{H}_{s+\varepsilon}-B^{H}_s)ds^{2H}
$$
in probability, where $f$ is a Borel function. In Yan {\em et al}~\cite{Yan7} one constructed the Banach space ${\mathscr H}=L^2({\mathbb R},\mu(dx))$ with
$$
\mu(dx)=\left(\int_0^Te^{-\frac{x^2}{2s^{2H}}}\frac{2Hds}{\sqrt{2\pi} s^{1-H}}\right)dx
$$
and
\begin{align*}
\|f\|_{\mathscr H}^2=\int_0^T\int_{\mathbb
R}|f(x)|^2e^{-\frac{x^2}{2s^{2H}}}\frac{2Hdxds}{\sqrt{2\pi}s^{1-H}} =E\left(\int_0^T|f(B^H_s)|^2ds^{2H}\right),
\end{align*}
such that the generalized quadratic covariation $[f(B^H),B^H]^{(H)}$ exists in $L^2(\Omega)$ and
\begin{equation}\label{sec6-eq6.00}
E\left|[f(B^H),B^H]_t^{(H)}\right|^2\leq C\|f\|^2_{\mathscr H},
\end{equation}
provided $f\in {\mathscr H}$. Moreover, the Bouleau-Yor identity takes the form
$$
[f(B^H),B^H]_t^{(H)}=-\int_{\mathbb {R}}f(x) {\mathscr L}^{H}(dx,t)
$$
for all $f\in {\mathscr H}$. By using the generalized quadratic covariation Yan {\em et al}~\cite{Yan7} obtained the next It\^o formula:
\begin{equation}\label{sec6-eq6.1}
F(B^H_t)=F(0)+\int_0^tf(B^H_s)\delta B^H_s+\frac12\left[f(B^H),B^H
\right]^{(H)}_t
\end{equation}
for all $0<H<\frac12$, where $F$ is an absolutely continuous function such that $F'=f\in {\mathscr H}$ is left (right) continuous. It is important to note that the method used in Yan {\em et al}~\cite{Yan7} is inefficacy for $\frac12<H<1$ in general and the similar results for $\frac12<H<1$ is unknown so far.
\begin{corollary}\label{cor6.1}
Let $0<H<\frac12$ and let $F(x)=x\log|x|-x$. Then $F'\in {\mathscr H}$ and the It\^o type formula
\begin{equation}\label{sec6-eq5.3}
F(B^H_t-a)=F(-a)+\int_0^t\log|B^H_s-a|\delta B^H_s +\frac12\left[\log(B^H-a),B^H\right]^{(H)}_t
\end{equation}
holds and
\begin{equation}\label{sec6-eq5.4}
{\mathcal C}^H_t(a)=\left[\log|B^H-a|,B^H\right]^{(H)}_t
\end{equation}
for all $t\geq 0$ and $a\in {\mathbb R}$.
\end{corollary}
\begin{proof}
Let $F_{+}$ and $F_{-}$ be defined in Section~\ref{sec3}. Then $F'_{+}\in {\mathscr H}$ is left continuous, and
\begin{equation}\label{sec6-eq5.3-00}
F_{+}(B^H_t-a)=F_{+}(-a)+\int_0^tF'_{+}(B^H_s-a)\delta B^H_s +\frac12\left[F'_{+}(B^H-a),B^H\right]^{(H)}_t
\end{equation}
by It\^o's formula~\eqref{sec6-eq6.1}. Similarly, we have
\begin{equation}\label{sec6-eq5.3-01}
F_{-}(B^H_t-a)=F_{-}(-a)+\int_0^tF'_{-}(B^H_s-a)\delta B^H_s +\frac12\left[F'_{-}(B^H-a),B^H\right]^{(H)}_t
\end{equation}
since $F'_{-}\in {\mathscr H}$ is right continuous. Thus, the corollary follows from $F=F'_{+}+F'_{-}$.
\end{proof}
Denote
\begin{align*}
\Theta_\varepsilon(t,a):&={\mathscr L}^H(a-\varepsilon,t)F'(-\varepsilon)
-{\mathscr L}^H(a+\varepsilon,t)F'(\varepsilon)\\
&=\left[{\mathscr L}^H(a-\varepsilon,t)-{\mathscr L}^H(a+\varepsilon,t)\right]\log\varepsilon.
\end{align*}
By integration by parts we have
\begin{align*}
{\mathcal C}^H_t(a)&=\left[\log|B^H-a|,B^H\right]^{(H)}_t=-\int_{\mathbb {R}}\log|x-a|{\mathscr L}^{H}(dx,t)\\
&=-\lim_{\varepsilon\downarrow 0}\left(\int_{a+\varepsilon}^\infty\log|x-a|{\mathscr L}^{H}(dx,t)
+\int_{-\infty}^{a-\varepsilon}\log|x-a|{\mathscr L}^{H}(dx,t)\right)\\
&=\lim_{\varepsilon\downarrow 0}\left(\int_{a+\varepsilon}^\infty\frac{{\mathscr L}^{H}(x,t)}{x-a}dx
+\int_{-\infty}^{a-\varepsilon}\frac{{\mathscr L}^{H}(x,t)}{x-a}dx\right)+\lim_{\varepsilon\downarrow 0}\Theta_\varepsilon(t,a)\\
&=\lim_{\varepsilon\downarrow 0}\int_{\mathbb R}1_{\{|x-a|\geq \varepsilon\}}\frac{{\mathscr L}^{H}(x,t)}{x-a}dx\\
&={\rm v.p.}\int_{\mathbb R}\frac{{\mathscr L}^{H}(x,t)}{x-a}dx=\pi {\mathscr H}{\mathscr L}^H(\cdot,t)(x)
\end{align*}
almost surely and in $L^2(\Omega)$, for all $t\geq 0$ and $a\in {\mathbb R}$ since $x\mapsto {\mathscr L}^H(x,\cdot)$ is H\"older continuous and has the compact support.
\begin{corollary}\label{lem6.1}
Let $F$ be given by Corollary~\ref{cor6.1} and let $g$ be a continuous function with compact support. Then the integral $F'\ast g\in {\mathscr H}$, and for all $0<H<\frac12$, the process
$$
2\left((F\ast g)(B^H_t)-(F\ast g)(0)-\int_0^t(F'\ast g)(B^H_s)\delta B^H_s\right)
$$
is well-defined in $L^2(\Omega)$ and is equal to
$$
[(F'\ast g)(B^H),B^H]^{(H)}_t
$$
for all $t\in [0,T]$.
\end{corollary}
Thus, similar to proof of Theorem~\ref{th5.3} we can obtain the following occupation formula.
\begin{theorem}\label{th6.2}
Let $0<H<\frac12$ and let $g$ be a continuous function with compact support. Then we have, almost surely,
\begin{equation}\label{sec6-eq5.26}
\int_{\mathbb R}{\mathcal C}^H_t(x)g(x)dx=2H\pi\int_0^t({\mathscr
H}g)(B^H_s)s^{2H-1}ds
\end{equation}
and
$$
2H\pi\int_0^tg(B^H_s)s^{2H-1}ds=\int_{\mathbb R}{\mathcal
C}^H_t(x)({\mathscr H}^{-1}g)(x)dx
$$
for all $t\in [0,T]$.
\end{theorem}
\begin{proof}
Let $F(x)=x\log|x|-x$. By Corollary~\ref{cor6.1} and~\eqref{sec6-eq6.00} we have
\begin{align*}
E&\left|\int_{\mathbb R}\left(\int_0^tF'(B^H_s-x)\delta B^H_s\right)g(x)dx\right|^2
<\infty,
\end{align*}
since $g$ admits a compact support. Thus, similar to Lemma~\ref{lem5.1} we can show that the Fubini theorem
\begin{equation}
\int_0^t\left(\int_{\mathbb R}F'(B^H_s-x)g(x)dx\right)\delta B^H_s
=\int_{\mathbb R}\left(\int_0^tF'(B^H_s-x)\delta B^H_s\right)g(x)dx
\end{equation}
holds for all $0\leq t\leq T$. It follows from Corollary~\ref{cor6.1} and Corollary~\ref{lem6.1} that
\begin{align*}
\frac12\int_{\mathbb R}{\mathcal C}^H_t(x)g(x)dx&=\int_{\mathbb R}\left(F(B^H_s-x)-F(-x)-\int_0^tF'(B^H_s-x)\delta B^H_s\right)g(x)dx\\
&=F\ast g(B^H_t)-F\ast g(0)-\int_0^tF'\ast g(B^H_s)\delta B^H_s\\
&=\frac12\left[(F'\ast g)(B^H),B^H\right]^{(H)}_t\\
&=-\frac12\int_{\mathbb R}(F'\ast g)(x){\mathscr L}^H(dx,t)
\end{align*}
On the other hand, by the H\"older continuity of $(x,t)\mapsto {\mathscr L}^H(x,t)$ and Lebesgue's dominated convergence theorem we have
\begin{align*}
\int_{\mathbb R}g(a)da&\int_{\mathbb R}F'(x-a){\mathscr L}^H(dx,t)\\
&=\int_{\mathbb R}g(a)\lim_{\varepsilon\downarrow 0 }\left(\int_{a+\varepsilon}^{\infty}F'(x-a){\mathscr L}^H(dx,t)+\int_{-\infty}^{a-\varepsilon}F'(x-a){\mathscr L}^H(dx,t)\right)da\\
&=\int_{\mathbb R}g(a)\lim_{\varepsilon\downarrow 0 }\left(\Theta_\varepsilon(t,a)-\int_{\mathbb R}1_{\{|x-a|>\varepsilon\}} F''(x-a){\mathscr L}^H(x,t)dx\right)da\\
&=-2H\int_{\mathbb R}g(a)\lim_{\varepsilon\downarrow 0 }\left(\int_0^t1_{\{|B^H_s-a|>\varepsilon\}} F''(B^H_s-a)s^{2H-1}ds\right)da\\
&=-2H\lim_{\varepsilon\downarrow 0 }\int_0^ts^{2H-1}ds\int_{\mathbb R}1_{\{|B^H_s-a|>\varepsilon\}}\frac{g(a)}{B^H_s-a}da\\
&=-2H\pi\int_0^t{\mathscr H}g(B^H_s)s^{2H-1}ds
\end{align*}
almost surely and in $L^2(\Omega)$, for all $t\in [0,T]$. This shows that
\begin{align*}
\frac12\int_{\mathbb R}&{\mathcal C}^H_t(x)g(x)dx=H\pi\int_0^t{\mathscr H}g(B^H_s)s^{2H-1}ds
\end{align*}
and the theorem follows.
\end{proof}

\begin{remark}
{\rm

When $0<H<\frac12$, from the discussion in this section, we have fund that for all non-locally integrable Borel functions $f\in {\mathscr H}$, the identities
$$
{\mathcal K}_t^{H}(f,a):=\lim_{\varepsilon\downarrow 0}\int_0^t1_{\{|B^H_s-a|>\varepsilon\}}f(B^H_s)ds^{2H}=[f(B^H),B^H]^{(H)}_t
$$
in $L^2(\Omega)$ (almost surely) and
$$
\int_{\mathbb R}{\mathcal K}_t^{H}(f,x)g(x)dx=2H\int_0^t{\rm v.p}(f'\ast g)(B^H_s)s^{2H-1}ds
$$
hold for all continuous functions $g$ with compact supports, provided
\begin{align*}
{\mathscr L}^H(a-\varepsilon,t)f(-\varepsilon)
-{\mathscr L}^H(a+\varepsilon,t)f(\varepsilon)\longrightarrow 0,
\end{align*}
in $L^2(\Omega)$ (almost surely), as $\varepsilon$ tends to zero.
}
\end{remark}


\begin{thebibliography}{99}
\bibitem{Nua1}
E. Al\'os, O. Mazet and D. Nualart, Stochastic calculus with respect to {G}aussian processes, {\it Ann. Prob.} {\bf 29} (2001), 766-801.

\bibitem{Bertoin0}
J. Bertoin, Sur une int\'egrale pour les processus \'a $\alpha$-variation born\'e, {\em Ann. Probab.} {\bf 17} (1989), 1521-1535.

\bibitem{Bertoin1}
J. Bertoin, Complements on the Hilbert transform and the fractional
derivative of Brownian local times, {\em J. Math. Kyoto Univ.} {\bf
30} (1990), 651-670.

\bibitem{Bertoin2}
J. Bertoin, Regularity of the Cauchy principal value of the local times of some L\'evy processes, {\em Bull. Sri. math.} {\bf 123}
(1999), 47-58.

\bibitem{BHOZ}
F.Biagini, Y. Hu, B. {\O}ksendal and T. Zhang, {\it Stochastic
calculus for fractional Brownian motion and applications,} Probability and its application, Springer, Berlin (2008).

\bibitem{Biane-Yor}
P. Biane and M. Yor, Valeurs principales associ\'ees aux temps locaux Browniens, {\em Bull. Sci. Math.} {\bf 111} (1987), 23-101.

\bibitem{Cherny}
A.S. Cherny, Principal values of the integral functionals of
Brownian motion: existence, continuity and an extension of It\^o'S
formula, {\em Lect. Notes Math.} {\bf 1755} (2001), 348-370.

\bibitem{Cout}
L. Coutin, D. Nualart and C. A. Tudor, Tanaka formula for the fractional Brownian motion, {\it Stochastic Process. Appl.} {\bf 94} (2001), 301-315.

\bibitem{Csaki1}
E. Cs\'aki, M. Cs\"org\"o, A. F\"oldes and Z. Shi, Increment sizes of the principal value of Brownian local time, {\em Probab. Theory Relat. Fields}, {\bf 117} (2000), 515-531.

\bibitem{Csaki2}
E. Cs\'aki, A. F\"oldes and Z. Shi, A joint functional law for the Wiener process and principal value, {\em Studia Sci. Math. Hungar.} {\bf 40} (2003), 213-241.

\bibitem{Csaki-Hu}
E. Cs\'aki and Yueyun Hu, On the Increments of the principal value
of Brownian local time, {\em Elect. J. Probab.} {\bf 10} (2005),
925-947.

\bibitem{Dec}
L. Decreusefond and A.S. \"Ust\"unel, Stochastic analysis of the
fractional Brownian motion, {\it Potential Anal.} {\bf 10} (1999), 177-214.

\bibitem{Eddahbi-Vives}
M. Eddahbi and J. Vives, Chaotic expansion and smoothness of some
functionals of the fractional Brownian motion, {\em J. Math. Kyoto Univ.}, {\bf 43} (2003), 349-368.

\bibitem{Fitzsimmons-Getoor1}
P. J. Fitzsimmons and R. K. Getoor, On the Distribution of the Hilbert Transform of the Local Time of a Symmetric Levy Process,
{\em Ann. Probab.} {\bf 20} (1992), 1484-1497.

\bibitem{Fitzsimmons-Getoor2}
P. J. Fitzsimmons and R. K. Getoor, Limit theorems and variation
properties for fractional derivatives of the local time of a stable process, {\em Ann. Inst. H. Poincar\'e Probab. Statist.}, {\bf 28} (1992), 311-333.

\bibitem{Geman}
D. Geman and J. Horowitz, Occupation densities, {\it Ann. Probab.}
{\bf 8} (1980), 1-67.

\bibitem{Grad2}
M. Gradinaru, F. Russo, P. Vallois, Generalized covariations, local time and Stratonovich It\^os formula for fractional Brownian motion with Hurst index $H\geq \frac14$, {\em Ann. Probab.} {\bf 31} (2003), 1772-820.

\bibitem{Hu1}
Yaozhong Hu, B. {\O}kesendal and D. M. Salopek, Weighted local time for fractional Brownian motion and applications to finance, {\it Stoch. Anal. Appl.} {\bf
23} (2005), 15-30.

\bibitem{Hu2}
Yaozhong Hu, Integral transformations and anticipative calculus for
fractional Brownian motions, {\it Memoirs Amer. Math. Soc.} {\bf Vol. 175} (2005), {\bf No. 825}.

\bibitem{Hu-Y}
Yueyun Hu, The laws of Chung and Hirsch for Cauchy's principal
values related to Brownian local times, {\em Elect. J. Probab.} {\bf 5} (2000), 1-16.

\bibitem{Ito-McKean}
K. Ito and H. P. McKean, {\em Diffusion processes and their sample paths}, Berlin, New York, Springer Verlag 1965.

\bibitem{King}
Frederick W. King, {\em Hilbert transforms}, Cambridge University Press 2012.

\bibitem{Mansuy-Yor}
R. Mansuy and M. Yor, {\em Aspects of Brownian motion}, Berlin,
Heidelberg, Springer-Verlag 2008.

\bibitem{Mishura2}
Y. S. Mishura, {\it Stochastic calculus for fractional Brownian motion and related processes}, Lect. Notes in Math. {\bf 1929} (2008).

\bibitem{Nourdin}
I. Nourdin, {\em Selected aspects of fractional Brownian motion}, Springer Verlag 2012.

\bibitem{Nua4}
D. Nualart,  {\it Malliavin calculus and related topics}, 2nd~edn.
Springer-Verlag 2006.

\bibitem{Pipiras}
V. Pipiras and M. Taqqu, Integration questions related to the
fractional Brownian motion, {\it Probab. Theory Related Fields},
{\bf 118} (2001), 251-281.

\bibitem{Tudor}
Ciprian A. Tudor, {\em Analysis of Variations for Self-similar Processes}, Springer-Verlag 2013.

\bibitem{Yamada1}
T. Yamada, On some representations concerning the stochastic
integrals, {\em Probab. Math. Statist.} {\bf 4} (1984), 153-166.

\bibitem{Yamada2}
T. Yamada, On the fractional derivative of Brownian local times, {\em J. Math. Kyoto Univ.} {\bf 25} (1985), 49-58.

\bibitem{Yamada4}
T. Yamada, Principal values of Brownian local times and their
related topics, {\em It\^o's stochastic calculus and probability
theory}, Springer, 413-422 (1996).

\bibitem{Yan8}
L. Yan, The fractional derivative for fractional Brownian local time, to appear in {\em Mathematische Zeitschrift} 2015.

\bibitem{Yan7}
L. Yan, J. Liu and C. Chen, The generalized quadratic covariation for fractional Brownian motion with Hurst index less than $1/2$, {\it Infin. Dimens. Anal. Quantum Probab. Relat. Top.} {\bf 17} No. 4, 2014 (32 pages).

\bibitem{Yan1}
L. Yan, J. Liu and X. Yang, Integration with respect to fractional local time with Hurst index $1/2<H<1$, {\it Potential Anal.}, {\bf 30} (2009), 115-138.

\bibitem{Yan10}
L. Yan and Q. Zhang, Hilbert transform of G-Brownian local time, {\em Stoch. Dyn.} {\bf 14} (2014), 1450006 (26 pages).

\bibitem{Yor1}
M. Yor, Sur la transform\'e de Hilbert des temps locaux browniens et une extension de la formule d'lt\^o, {\em Lect. Notes Math.}, {\bf
920} (1982), 238-247.

\bibitem{Yor2}
M. Yor (editor), {\em Exponential Functionals and Principal Values Related to Brownian Motion}, Biblioteca de la Revista Matem\'atica Iberoamericana, Madrid 1997.

\end{thebibliography}
\end{document}